\documentclass[reqno,centertags, 11pt]{amsart}
\usepackage{comment,graphicx}
\usepackage{color}

\definecolor{purple}{rgb}{0.65, 0, 0.9}
\definecolor{orange}{rgb}{1,.5,0}

\usepackage{graphics}
\usepackage{amssymb,amsmath,amsfonts}%numbysec}
-\marginparwidth \oddsidemargin -9mm \evensidemargin \oddsidemargin
\usepackage{latexsym}
\advance\hoffset by 5mm

\def\@abssec#1{\vspace{.05in}\footnotesize \parindent .2in
{\bf #1. }\ignorespaces}
\newtheorem{theorem}{Theorem}[section]

\newtheorem{lemma}[theorem]{Lemma}
\newtheorem{proposition}[theorem]{Proposition}

\newtheorem{remark}[theorem]{Remark}

\def \R {\mathbb R}

\def \N {\mathbb N}

\newcommand{\be}{\mathbf e}

\allowdisplaybreaks \numberwithin{equation}{section}

\renewcommand{\be}{\begin{equation}}
\newcommand{\ee}{\end{equation}}

\title[Two-species competition model with chemotaxis]{Two-species competition model with chemotaxis:
well-posedness, stability and dynamics}
\author{Guanlin Li}

\author{Yao Yao}
\date{}

\begin{document}

\maketitle

\begin{abstract}We study a system of PDEs modeling the population dynamics of two competitive species whose spatial movements are governed by both diffusion and \emph{mutually repulsive} chemotaxis effects. We prove that solutions to this system are globally well-posed, without any smallness assumptions on the chemotactic coefficients. Moreover, in the weak competition regime, we prove that neither species can be driven to extinction as the time goes to infinity, regardless of how strong the chemotaxis coefficients are. Finally, long-time behaviors of the system are studied both analytically in the weakly nonlinear regime, and numerically in the fully nonlinear regime.
\end{abstract}

\section{Introduction}

\subsection{Background}In this paper, we consider two biological species that compete for resources, with their spatial movements governed by diffusion and \emph{mutually repulsive} chemotaxis effects. 
 Let us denote by $u(x,t)$ and $v(x,t)$ the population densities of the two species, and let the spatial domain be a bounded domain $\Omega \subset \mathbb{R}^n$ with smooth boundary, where $n\geq 1$. 
 The time evolution of $(u(x,t),v(x,t))$ is governed by the following three effects: competition, diffusion, and chemotactic movements.  Before stating the full system of equations, we will discuss each of the three effects  and briefly review the relevant literature.

(1) \emph{Competition between the two species}. The competition effect will be modeled by a Lotka--Volterra-type system \cite{Lotka, volterra} for competing species. If $u$ and $v$ are spatially homogeneous (thus they are only functions of $t$),  after a normalization by suitable multiplicative constants, their time evolution is described by the following system of ODEs
\begin{equation}\label{ode_system}
\begin{cases}
\dot{u}(t) = b_1 u(1-u-a_1 v),\\
\dot{v}(t) = b_2 v (1-v-a_2 u),
\end{cases}
\end{equation}
where $a_1, a_2, b_1, b_2 > 0$. It is well-known (see \cite[Section 12]{Kot} for example) that solutions to \eqref{ode_system} exhibit different types of asymptotic behavior depending on the coefficients $a_1$ and $a_2$.  Namely, in the \emph{weak competition} regime $a_1,a_2 \in (0,1)$, there is a unique positive steady state
\begin{equation}\label{steady_general}
(\overline u, \overline v) = \left(\frac{1-a_1}{1-a_1a_2}, \frac{1-a_2}{1-a_1 a_2}\right),
\end{equation}
and it is the global attractor of all solutions with positive initial data. This indicates that the two species co-exist in the long run. In the \emph{strong competition} regime $a_1>1$, $a_2>1$, \eqref{steady_general} is still the unique positive steady state, however it is unstable, and for a generic positive initial data, the solution either converges to $(0,1)$ or $(1,0)$ as $t\to\infty$, depending on the initial data. Finally, if $0<a_1<1<a_2$, there is no positive steady state, and all positive solutions converge to $(1,0)$. Likewise, if $0<a_2<1<a_1$ all positive solutions converge to $(0,1)$.

(2) \emph{Spatial diffusion effects}. If the species have random spatial movements, this effect is typically modeled  by diffusion terms $d_1 \Delta u$ and $d_2\Delta v$ in the evolution for $u$ and $v$, where $d_1, d_2>0$. When both diffusion and competition effects are taken into account, the time evolution of $(u(x,t),v(x,t))$ satisfies the following system of reaction-diffusion equations:
\begin{equation}\label{reaction_diffusion}
\begin{cases}
\partial_t u = d_1\Delta u + b_1 u(1- u- a_1 v) & \text{ in } \Omega\times(0,T),\\
\partial_t v = d_2 \Delta v + b_2 v(1 -v - a_2 u) & \text{ in } \Omega\times(0,T),\\
\frac{\partial u}{\partial \nu} = \frac{\partial v}{\partial \nu} = 0 & \text{ on } \partial\Omega\times(0,T),
\end{cases}
\end{equation}
where $\frac{\partial}{\partial \nu}$ is the outward normal derivative at $\partial\Omega$, and the Neumann boundary condition indicates that both densities have zero flux at the boundary. 
This system has been extensively studied by both mathematicians and ecologists. In the weak competition regime, it is known that the spatially homogeneous positive steady state \eqref{steady_general} is the global attractor for any positive initial population \cite{CHS, MR}, and there are no non-constant steady states. In the strong competition regime things are more complicated: in addition to the (unstable) steady state \eqref{steady_general} and the stable constant steady states $(0,1)$ and $(1,0)$, the system can have some stable spatially inhomogeneous steady states for some non-convex $\Omega$ (e.g. dumbbell shaped) when the diffusion coefficients $d_1, d_2$ are not too large \cite{MM1,  MEF}.  See \cite{CTV, Dancer, EFL, GL} for results on co-existence steady states of \eqref{reaction_diffusion} with Dirichlet boundary conditions. When $\Omega=\R$ is the whole real line, traveling wave solutions to \eqref{reaction_diffusion} have been well-studied \cite{ Gardner,   Kan1, MT, RM, TF}. For further results on this system and related models (e.g. non-homogeneous environment, cross diffusion), see \cite{BDPS, CC, CJ04, CJ06, DLM,  LN1, LN2, MK, SKT} and the references therein.

(3) \emph{Chemotactic movements.} Besides the above two effects, the two species may also experience chemotaxis effects in response to certain chemicals. For a single species, the most classical model for chemotaxis is the Keller-Segel equation \cite{KS}. It describes the collective motion of cells (slime mold or certain bacteria such as E. coli) with density $\rho(x,t)$ that are attracted by a chemical substance with density $c(x,t)$. The chemical attractant is emitted by the cells themselves, and it also diffuses in space and decays in time. This leads to the following system of parabolic equations
\begin{equation}\label{KS}
\begin{cases}
\rho_t = \Delta \rho - \chi\nabla\cdot (\rho \nabla c) & \text{ in } \Omega \times (0,T)\\
\varepsilon c_t = \Delta c -  c + \rho & \text{ in } \Omega \times (0,T)\\
\frac{\partial \rho}{\partial \nu} = \frac{\partial c}{\partial \nu} = 0  & \text{ on } \partial\Omega\times(0,T).
\end{cases}
\end{equation}
Since the chemical attractant reaches its equilibrium in a much faster timescale, it is common to assume that $c$ is quasi-stationary by taking the $\varepsilon\to 0$ limit, so that the second equation of \eqref{KS} is replaced by the elliptic equation
\begin{equation}\label{c_rho}
 c - \Delta c = \rho,
\end{equation}
and \eqref{KS} becomes a parabolic-elliptic system. Since the primary interest of this work is on two-species chemotaxis models, we will not discuss the vast literature on the Keller-Segel equation. We refer the readers to the surveys \cite{BBTW, hillen, horstmann, horstmann1} and the references therein. Below we will state the two-species system with all three effects, and introduce some relevant literature on similar models.

\subsection{Two-species system with competition, diffusion, and chemotaxis}
%We are now ready to state the system of equations we study. 
In addition to the diffusion and competition effects in \eqref{reaction_diffusion}, we further assume that $u$ and $v$ each emits some chemical substance with their concentrations denoted by $c[u]$ and $c[v]$ respectively, and each species exhibits chemotactic movements driven by the chemical emitted by the other species. The new system with chemotaxis reads
%For a bounded domain $\Omega \subset \mathbb{R}^n$ with smooth boundary, $u = u(x,t)$ and $v(x,t)$ represent the population density of the two species.
\begin{equation}\label{simplified}
\begin{cases}
\partial_t u  = d_1\Delta u +   \chi_1 \nabla\cdot( u \nabla c[v]) + b_1 u(1 - u - a_{1} v) &\text{ in }\Omega \times (0,T), \\
\partial_t v =  d_2\Delta v +   \chi_2 \nabla\cdot( v \nabla c[u]) + b_2 v(1 - v - a_{2} u) &\text{ in }\Omega \times (0,T),  \\
\frac{\partial u}{\partial \nu} = \frac{\partial v}{\partial \nu} = 0 & \text{ on } \partial \Omega  \times (0,T),
\end{cases}
\end{equation}
where the sign and magnitude of the chemotactic coefficients $\chi_1,\chi_2$ indicate the type (attractive or repulsive) and strength of the chemotactic movements. 
For a population density $\rho$ (here $\rho$ is either $u$ or $v$), following \eqref{c_rho} (which comes from the $\varepsilon\to 0$ limit of the second equation in \eqref{KS}), we assume that the chemical concentration $c[\rho]$ solves the elliptic equation
\begin{equation}\label{poisson}
\begin{cases}
c - \Delta c = \rho &\text{ in } \Omega\\
\frac{\partial c}{\partial \nu} = 0 & \text{ on } \partial \Omega. 
\end{cases}
\end{equation}

Note that in the system \eqref{simplified}, the drift velocity for $u$ and $v$ are given by $-\chi_1 \nabla c[v]$ and $-\chi_2 \nabla c[u]$ respectively.
Since $u$ and $v$ are competing species, throughout this paper we assume that $\chi_1>0$ and $\chi_2>0$, indicating that both species tend to move away from the chemical emitted by the other species (i.e. $u$ and $v$ are \emph{mutually repulsive}). The goal  of this paper is to investigate the influence of the mutually repulsive chemotactic effect on the solutions, in particular, whether it affects the global well-posedness  and long time behaviors of solutions.

In the past two decades, two-species chemotaxis models have attracted much interest. Many previous works deal with diffusion and chemotaxis effects only (i.e. the competition terms are absent), where the mass of each species is conserved. In this case, the system often has an associated monotone energy functional (as in the Keller--Segel equation), and many analysis results are based on this variational structure. When multiple species are mutually attracted to each other by reacting to different chemicals, \cite{wolansky} studied global existence of solutions and equilibria using the energy functional.  When both species are attracted by the same chemical substance produced by themselves, \cite{ESV} identified some conditions that lead to global existence and finite time blow-up, and later a sharp condition in $\R^2$ was established in \cite{CEV}, and blow-up examples for $n>2$ were constructed in \cite{BEG}. For a general multi-species chemotaxis model with attraction and repulsion between the species, the steady states of the system and their stability were studied by  \cite{horstmann2}. For the system \eqref{simplified}--\eqref{poisson} without the competition terms, various criteria for global well-posedness and blow-up were derived in \cite{TW2}. When the two-species system has nonlocal interaction, well-posedness of solutions via gradient flow theory and properties of steady states are studied by \cite{DF, DF2, MKB}. See \cite{BDFS, CFSS, CFS, CHS18,  DEF, DF} for studies on two-species systems with nonlocal interaction and cross-diffusion.

In the presence of all the three effects (diffusion, chemotaxis, and competition), note that the system does not have any variational structure. When there is only one chemical $c$ (where $c$ is generated/consumed by both $u$ and $v$ in the sense that $\varepsilon c_t = \Delta c - c \pm (u + v)$), the parabolic-elliptic system (with $\varepsilon=0$) was studied by \cite{Mizukami2, STW, TW}, and the corresponding parabolic-parabolic systems (with $\varepsilon>0$) were studied by \cite{BW, BLM, LMW, Mizukami, WMHZ, WYZ, WZYH}. In these works, conditions that lead to global well-posedness were established, and  parameter regimes for which the constant steady state \eqref{steady_general} is asymptotically stable were also identified.

More recently, the system \eqref{simplified}--\eqref{poisson} with two chemicals $c[u]$ and $c[v]$ has been studied in the following works. \cite{ZLY} considered the mutually attractive case (i.e. $\chi_1,\chi_2<0$), and showed that when $-\chi_1,-\chi_2$ are below certain thresholds, the system is globally well-posed and converges towards the steady state \eqref{steady_general}. \cite{CNT} showed that when $|\chi_1|,|\chi_2|$ are below certain thresholds, \eqref{simplified}--\eqref{poisson} has a global solution, and the solution converges towards \eqref{steady_general} in the weak competition regime, and converges towards $(1,0)$ if $0<a_1<1<a_2$. \cite{WM} studied the mutually attractive case, and showed that solutions converge towards constant steady states when $b_1, b_2$ are sufficiently large. \cite{WZMH} studied the system when $c[u]$ and $c[v]$ depends on $u,v$ in a nonlinear way.  Let us also point out that for the corresponding parabolic-parabolic system,  conditions that leads to global well-posedness and finite-time blow-up were established in \cite{LX, PWZW, RL}.

\subsection{Our results.}
Our goal is to study the well-posedness, stability of steady state, and the dynamics of the system \eqref{simplified}--\eqref{poisson} in the \emph{mutually repulsive} case $\chi_1,\chi_2>0$. We aim to explore the following questions:

\begin{enumerate}
\item[(a)] For any $\chi_1,\chi_2>0$, is the solution $(u,v)$ to \eqref{simplified}--\eqref{poisson} always globally well-posed in time?
\smallskip
\item[(b)] When $\chi_1,\chi_2>0$ and the system is in the weak competition regime $a_1, a_2\in(0,1)$, is it possible that one species becomes extinct as $t\to\infty$?
\smallskip
\item[(c)] What is the long-time behavior of \eqref{simplified}--\eqref{poisson}?
\end{enumerate}

Our first result is as follows, which gives a positive answer to question (a). It shows that in the mutually repulsive case $\chi_1, \chi_2>0$, regardless of how large $\chi_1,\chi_2$ are, the system \eqref{simplified}--\eqref{poisson} is globally well-posed in all dimensions $n\geq 1$, and any solution is globally bounded in time. Recall that  global well-posedness has been previously established by \cite{CNT, ZLY} under some smallness assumptions of $|\chi_1|$ and $|\chi_2|$. Also, for any $\chi_1,\chi_2>0$, global well-posedness for the system \eqref{simplified}--\eqref{poisson} without the competition terms was obtained in \cite{TW2} in dimensions $n\leq 3$.

%The first question is whether the system is globally well-posed. 

%
%
%Previous results:
%
%* If $|\chi_1|, |\chi_2|< M$ : Globally well-posed.
%* If $b_1, b_2=0$ and $\chi_1, \chi_2 > 0$: Globally well-posed for $1\leq n \leq 3$.
%

\begin{theorem}\label{gwp}
For $n\geq 1$, let $\Omega \subset \R^n$ be a bounded domain with smooth boundary, and assume $\chi_1,\chi_2, a_1,a_2,b_1,b_2>0$.  For any non-negative initial data $u_0, v_0 \in C(\overline\Omega)$, the unique classical solution $(u,v)$ to \eqref{simplified}--\eqref{poisson} exists globally in time, where $u,v \in C^{2,1}(\Omega \times (0,\infty)) \cap C(\overline{\Omega} \times [0,\infty)) $, and $u,v$ are uniformly bounded above for all times.

In addition, for any $t_0>0$, the equation has the instant regularization effect, in the sense that
\begin{equation}\label{estimate_linfty}
\sup_{t>t_0}\left( \|u(\cdot, t)\|_{L^\infty} + \|v(\cdot, t)\|_{L^\infty}\right) \leq g(t_0),
\end{equation} where $g(t_0)$ only depends on $t_0, \Omega, n$ and the coefficients of \eqref{simplified}, and in particular does not depend on the initial data $u_0, v_0$. 
\end{theorem}

The proof of Thorem~\ref{gwp} is given in Section~\ref{sec_thm1}. The proof is based on an energy estimate, where the key observation is that the mutually repulsive nature of the two species allows us to obtain a priori $L^p$ bounds on $u,v$ that are uniformly bounded in time for any $1\leq p<\infty$.

Next we move on to question (b) that concerns the \emph{weak competition regime} $a_1, a_2 \in (0,1)$. In this regime,  for small $|\chi_1|, |\chi_2|$, \cite{CNT} showed that the constant steady state \eqref{steady_general} is the global attractor for the system \eqref{simplified}--\eqref{poisson}. Clearly we do not expect this property to hold for large $\chi_1, \chi_2$:  a simple linear stability criterion in Section~\ref{linear_stability} shows that \eqref{steady_general} is linearly unstable for large $\chi_1,\chi_2$. Indeed, numerics in Section~\ref{sec_simulations} suggest that solutions can form many patterns for large $\chi_1,\chi_2$, including some steady states where one species has a much smaller mass than the other (see Figure~\ref{Fig: fig_1D_fullnonlinear_SteadyState}(b)). It is then a natural question whether one of the species could be driven to extinction as $t\to\infty$ for large $\chi_1, \chi_2$. 
Our second result shows that this can never happen in the weak competition regime for any $\chi_1,\chi_2>0$,  giving a negative answer to question (b). 

\begin{theorem}\label{no_extinction}For $n\geq 1$, let $\Omega \subset \R^n$ be a bounded domain with smooth boundary, and assume $\chi_1,\chi_2, a_1,a_2,b_1,b_2>0$. Assume that $a_1,a_2$ are in the weak competition regime, that is, $a_1, a_2 \in (0,1)$. For any initial data $u_0(x), v_0(x) \in C(\overline\Omega)$ that are not identically zero, the solution $(u,v)$ to \eqref{simplified}--\eqref{poisson} satisfies 
\begin{equation}\label{goal_l1}
\|u(t)\|_{L^1} \geq \delta \min\{\|u(1)\|_{L^1}, 1\} \text{ and } \|v(t)\|_1 \geq \delta \min\{\|v(1)\|_{L^1}, 1\} \quad\text{ for all }t\geq 1,
\end{equation}
where $\delta>0$ only depends on $\Omega, n$ and the coefficients of \eqref{simplified}, and in particular is independent of $u_0, v_0$.
\end{theorem}

Note that \eqref{goal_l1} gives a lower bound of the mass of $u(\cdot, t),v(\cdot,t)$ for all $t\geq 1$ (where $0<\delta\ll 1$ is a constant that is independent of the initial data), which implies that neither species can be driven to extinction. We point out that it is necessary to let the right hand side of \eqref{goal_l1} dependent on the mass of $u(\cdot,1),v(\cdot,1)$ (rather than $u_0$, $v_0$), since an initial data $u_0$ with mass 1 but very spiky profile (close to a delta function) can have its mass drop to an arbitrarily small number after a very short time, due to the $-b_1 u^2$ reaction term in \eqref{simplified}.

Our next results concern the question (c). We first consider the case when $\chi_1=\chi_2$ (call it $\chi$), and obtain a linear stability criteria for the constant steady state \eqref{steady_general} in terms of $\chi$ in Section~\ref{linear_stability}. Namely, there is a constant $\chi^*$, such that \eqref{steady_general} is linearly stable when $0<\chi<\chi^*$, and linearly unstable when $\chi>\chi^*$. In the \emph{weakly nonlinear regime}, i.e. when $\chi= \chi^*+\varepsilon$ with $0<\varepsilon\ll 1$, we use  a  standard  perturbative  expansion approach to derive the amplitude equation for our system in Section~\ref{sec_weakly_nonlinear}, and compared the analytical solution to the numerical solution in Section~\ref{sec_num_amplitude}. 

Finally, in Section~\ref{sec_simulations} we turn to the fully nonlinear regime with large $\chi_1,\chi_2$, %Since the constant steady state \eqref{steady_general} is linearly unstable, it is natural to expect some non-constant steady states. It is a challenging question to prove this rigorously, due to the lack of variational structure in this system caused by both non-locality and nonlinearity. 
and perform extensive numerical simulations that suggest a variety of long-time behaviors in one dimension. Some behaviors are commonly observed in other two-species population dynamics models: for example, when $\chi_1,\chi_2$ are large, small perturbations to the (linearly unstable) constant steady state \eqref{steady_general} leads to periodic patterns (see Figure~\ref{Fig: fig_1D_fullnonlinear_SpatialTime_profile} and \ref{Fig: fig_1D_fullnonlinear_SteadyState}). Also, if $\Omega= \mathbb{R}$ is the whole real line, initially segregated initial data lead to traveling wave solutions with one species invading the other (see Figure~\ref{Fig: fig_1D_traveling_wave_SpatialTime} and \ref{Fig: fig_1D_traveling_wave_profiles}), similar to those found in  \cite{MT} for the system \eqref{simplified} in the absence of chemotaxis. But the chemotaxis effect also brings us some unexpected patterns:  when $u_0,v_0$ are both compactly supported, for large $\chi_1,\chi_2$, we noticed that the two densities form clusters and propagates outwards in turn (see Figure~\ref{Fig: fig_1D_rippling_wave}). In Section~\ref{num_2d}, we also show two examples of 2D numerical simulations in the weakly nonlinear regime and fully nonlinear regime respectively.

\subsection{Organization of the paper} In Section~\ref{sec2}, we prove  Theorem~\ref{gwp} that gives the global well-posedness result for any $\chi_1,\chi_2>0$ and any dimension $n\geq 1$, as well as Theorem~\ref{no_extinction}, which shows neither species could extinct in the weak competition regime. Section~\ref{sec3} is dedicated to the stability of the constant steady state \eqref{steady_general}: we first identify the linear stability criteria, then derive the amplitude equation in the weakly nonlinear regime. In Section~\ref{sec_simulations} we perform numerical simulations in both one dimension and two dimensions, showing a variety of long-time behaviors of the system \eqref{simplified}--\eqref{poisson}. Finally, in Appendix~\ref{sec_numerical} we briefly describe the semi-implicit finite volume scheme we use in the numerical simulations. 

\noindent {\bf Acknowledgement.} \rm YY was partially supported by NSF grants DMS-1715418, 1846745 and the Sloan Research Fellowship. YY would like to thank Jos\'e A. Carrillo and Lenya Ryzhik  for helpful discussions.

\section{Global well-posedness and quantitative estimates}\label{sec2}

\subsection{Global well-posedness of solutions}\label{sec_thm1}
In this subsection we aim to prove Theorem~\ref{gwp}, which gives the global well-posedness for the system \eqref{simplified}--\eqref{poisson} in any dimension $n\geq 1$ when the two species are mutually repulsive (that is, $\chi_1,\chi_2>0$).
Given a pair of non-negative initial condition $u_0, v_0 \in C(\overline{\Omega})$ with $u_0, v_0 \not\equiv 0$, local well-posedness of \eqref{simplified}--\eqref{poisson}  follows from an almost identical argument as \cite[Lemma 2.1]{STW}, which we do not repeat here. Namely,  there exists a maximal time $T_{\max}> 0$ and a unique pair of classical solution $(u,v) \in C(\overline\Omega \times (0,T_{\max})) \cap C^{2,1}(\Omega \times (0,T_{\max}))$. Furthermore, $u,v$ are both positive during their existence, and the following blow-up criterion holds: if $T_{\max} < \infty$, then $\lim\sup_{t\nearrow T_{\max}} (\|u(t)\|_{L^\infty} + \|v(t)\|_{L^\infty}) = \infty$.

Before the proof, we collect some results regarding the solution $c$ to \eqref{poisson}.
\begin{lemma}\label{lemma_c}
Let $f \in L^p(\Omega)$ for some $1\leq p< \infty$, and let $c[f]$ be the solution to \eqref{poisson}. Then 
$c[f] \in W^{2,p}(\Omega)$ with 
\begin{equation}\label{w2p}
\|c[f]\|_{W^{2,p}(\Omega)} \leq C\|f\|_{L^p},
\end{equation} where $C$ only depends on $n, p$ and $\Omega$. In addition, if $f\geq 0$, then 
\begin{equation}\label{f_L1}
\|c[f]\|_{L^1} = \|f\|_{L^1},
\end{equation} and
\begin{equation}\label{max_min_f}
0 \leq c[f] \leq \sup_\Omega f
\end{equation} if $f$ is also bounded above.
\end{lemma}

\begin{proof}
\eqref{w2p} follows from the standard elliptic regularity estimate (e.g. \cite[pp. 194]{krylov}). If $f\geq 0$, then \eqref{max_min_f} follows from the maximum principle. Integrating $c - \Delta c = f$ in $\Omega$ and using the Neumann boundary condition and the facts that $f, c[f]\geq 0$ directly yields \eqref{f_L1}. 
\end{proof}

Now we are ready to prove Theorem~\ref{gwp}.

\begin{proof}[\textup{\textbf{Proof of Theorem~\ref{gwp}}}] Due to the blow-up criterion that we discussed at the beginning of this section, in order to show that the system \eqref{simplified}--\eqref{poisson} is globally well-posed, it suffices to show that both densities has a uniform-in-time upper bound.  The proof is split into the following 3 steps.

\noindent\textbf{Step 1. Estimates on the mass.} We start with a simple observation that the masses of $u$ and $v$ are uniformly bounded in time during the existence of a solution. To see this, let $M_u(t) := \int_\Omega u(x,t) dx$ be the mass of $u$ at time $t$, and note that $M_u(t)=\|u(t)\|_{L^1}$ since $u\geq 0$ for all times. Taking its time derivative and apply the divergence theorem, we have 
\begin{equation}\label{dm_temp}
\frac{d}{dt}M_u(t) = \int_{\Omega} b_1 u(1 - u - a_{1} v) dt \leq b_1 M_u(t) - b_1 \int_\Omega u(x,t)^2 dt \leq b_1 M_u \left(1 - \frac{M_u}{|\Omega|}\right),
\end{equation}
implying that 
\[
M_u(t) \leq \max\{M_u(0),|\Omega|\}\quad\text{ for all }t\geq 0,
\]
and a similar estimate can be derived for $M_v(t):= \int_\Omega v(x,t) dx$. Summing up the estimates,  the total mass $M(t) := \int_\Omega u(x,t)+ v(x,t) dx$ satisfies
\begin{equation}\label{estimate_m1}
M(t) \leq \|u_0\|_{L^1}+\|v_0\|_{L^1}+ 2|\Omega|, 
\end{equation}
which is uniformly bounded in time. 

Next we will show that for any $t_0>0$, $\sup_{t>t_0} M(t)$ is independent of the initial data. Note that \eqref{dm_temp} gives $M_u'(t) \leq -\frac{b_1}{2|\Omega| } M_u(t)^2$ whenever $M_u(t)>2|\Omega|$, and solving this differential inequality gives the following upper bound of $M_u(t)$ that is independent of the initial data (however it becomes unbounded as $t\to 0^+$): 
\begin{equation}\label{mu_2}
M_u(t) \leq \max\left\{ \frac{2|\Omega|}{b_1} t^{-1} , 2|\Omega|\right\}.
\end{equation}
A similar estimate can be derived for $M_v$, with $b_1$ replaced by $b_2$. Summing them up yields
$M(t) \leq 2|\Omega| ((b_1^{-1} + b_2^{-1}) t^{-1} + 2).$
Thus for any $t_0>0$, we have 
\begin{equation}\label{m_bound2}
\sup_{t>t_0} M(t) \leq 2|\Omega| ((b_1^{-1} + b_2^{-1}) t_0^{-1} + 2) =: g_{M}(t_0),
\end{equation} where $g_{M}(t_0)$ only depends on $t_0, \Omega, b_1, b_2$, and in particular does not depend on the initial data $u_0, v_0$. (However note that $g_M(t_0)\to \infty$ as $t_0\to 0^+$.)

\noindent\textbf{Step 2. $L^p$ estimates.} For any $1< p < \infty$, let 
\[
F_p(t) := \int_\Omega u(x,t)^p + v(x,t)^p \,dx.
\] In this step, we will show that $F_p(t)$ is uniformly bounded in time for any $1<p<\infty$, where the bound depends on $u_0,v_0, p, \Omega$ and the coefficients of \eqref{simplified}.  In addition, for any $t_0>0$, we will also show $\sup_{t>t_0} F_p(t) \leq g_p(t_0)$ for some $g_p(t_0)$ independent of $u_0, v_0$.

Taking the time derivative of $\int_\Omega u^p dx$ and applying the divergence theorem, we have
\begin{align}
\frac{d}{dt} \int_\Omega u^p dx &= \nonumber
\int_\Omega pu^{p-1} \big(d_1 \Delta u +   \chi_1 \nabla\cdot( u \nabla c[v]) + b_1 u(1 - u - a_{1} v) \big) dx\\
&= -\frac{4(p-1)d_1}{p} \int_\Omega \left|\nabla u^{\frac{p}{2}}\right|^2 dx + (p-1)\chi_1 \int_\Omega u^p \Delta c[v] dx + pb_1 \int_\Omega u^p(1-u-a_1v) dx \nonumber\\
&\leq (p-1)\chi_1 \int_\Omega u^p c[v] dx + pb_1 \int_\Omega u^p dx - pb_1 \int_\Omega u^{p+1} dx \nonumber\\
&=: I_1 + I_2 + I_3, \label{du2}
\end{align}
where in the inequality we used $u,v\geq 0$ as well as  $\Delta c[v] = c[v] -v \leq c[v]$.

Among the three terms on the right hand side of \eqref{du2}, $I_3$ is always negative, thus it suffices to control $I_1$ and $I_2$. Next we aim to control $I_1$ using $\int u^{p+1} dx$, $\int v^{p+1} dx$ and $\|v\|_{L^1}$. Applying H\"older's inequality and Young's inequality, we have
\[
I_1 \leq (p-1)\chi_1 \left(\int_\Omega u^{p+1} dx\right)^\frac{p}{p+1} \|c[v]\|_{L^{p+1}} \leq \frac{pb_1}{4} \int_\Omega u^{p+1} dx + C \|c[v]\|_{L^{p+1}}^{p+1},
\]
where $C$ only depends on $p$ and the coefficients of \eqref{simplified}. (In the following, the value of $C$ may vary from line to line.) Using the Gagliardo--Nirenberg inequality, we control $\|c[v]\|_{L^{p+1}}$ as
\[
\|c[v]\|_{L^{p+1}} \leq C \|c[v]\|_{W^{2,{p+1}}}^\theta \|c[v]\|_{L^1}^{1-\theta} \leq C\|v\|_{L^{p+1}}^\theta \|v\|_{L^1}^{1-\theta},
\]
where $\theta = \frac{np}{np+2(p+1)} < 1$, and the last inequality follows from \eqref{w2p} of Lemma~\ref{lemma_c}. Plugging this into the estimate of $I_1$ gives
\begin{equation}\label{I1}
I_1 \leq \frac{pb_1}{4} \int_\Omega u^{p+1} dx + C \|v\|_{L^{p+1}}^{(p+1)\theta} \|v\|_{L^1}^{(p+1)(1-\theta)} \leq \frac{pb_1}{4}\int_\Omega u^{p+1} dx + \frac{pb_2}{4}\int_\Omega v^{p+1} dx + C \|v\|_{L^1}^{p+1},
\end{equation}
where the last inequality follows from Young's inequality. Applying this estimate to \eqref{du2}, we have
\begin{equation}\label{u_final}
\frac{d}{dt} \int_\Omega u^p dx\leq pb_1 \int_\Omega u^p dx - \frac{3 pb_1}{4} \int_\Omega u^{p+1} dx + \frac{pb_2}{4}\int_\Omega v^{p+1} dx + C M(t)^{p+1}.
\end{equation}

Using an identical argument as above, the evolution of $\int_\Omega v^p dx$ satisfies a similar differential inequality as \eqref{u_final}, except that $u$ and $v$ become interchanged, and so do $b_1$ and $b_2$. Adding it with \eqref{u_final} gives
\[
\frac{d}{dt} F_p(t) \leq p (b_1+b_2) F_p(t)  -\frac{pb_1}{2}\int_{\Omega}u^{p+1} dx  -\frac{pb_2}{2}\int_{\Omega}v^{p+1} dx  + C M(t)^{p+1}. \]
Note that H\"older's inequality gives
\[
\int_\Omega u^p dx \leq \left(\int_\Omega u^{p+1}dx\right)^{\frac{p}{p+1}} |\Omega|^{\frac{1}{p+1}},
\]
and applying it to above yields
\begin{equation}\label{F_temp}
\frac{d}{dt} F_p(t) \leq  - c F_p(t)^{\frac{p+1}{p}} + C F_p(t) + C M(t)^{p+1},
\end{equation}
where $c$ and $C$ are positive constants only depending on $p, \Omega$ and the coefficients of \eqref{simplified}. Since $M(t)$ is uniformly bounded in time by \eqref{estimate_m1}, we have that $F_p(t)$ is also uniformly bounded in time, where the bounds depends on $u_0, v_0, p, \Omega$ and the coefficients of \eqref{simplified}.

In addition, recall that for any $t_0>0$, \eqref{m_bound2} gives that $\sup_{t>t_0} M(t) \leq g_M(t_0)$ which does not depend on $u_0, v_0$. Thus for all $t>t_0$, $F_p$ satisfies the differential inequality \eqref{F_temp} with all $M$ replaced by $g_M(t_0)$, which implies that $F_p'\leq -c F_p^{\frac{p+1}{p}}$  if $F_p\gg \max\{g_M(t_0)^p, 1\}$. Solving this differential inequality, we have $\sup_{t>t_0} F(t) \leq g_p(t_0)$ for some decreasing function $g_p(t_0)$ with $\lim_{t_0\to 0^+} g_p(t_0)=\infty$, where $g_p$  depends on $p, t_0, \Omega$ and the coefficients of \eqref{simplified}, but does not depend on the initial data.

\noindent\textbf{Step 3. $L^\infty$ estimates.} Once we obtain the uniform-in-time $L^p$ estimates for all $p\in(1,\infty)$, the $L^\infty$ estimates follows from a standard comparison principle argument. Using that $\Delta c[v] = c[v]-v$ and $u,v\geq 0$, the equation for $u$ satisfies the following (where we treat $v$ and $c[v]$ as a priori given functions):
\[
\begin{split}
\partial_t u &= d_1\Delta u + \chi_1\nabla u \cdot \nabla c[v] + \chi_1 u (c[v]-v) + b_1 u(1-u-a_1v)\\
&\leq d_1\Delta u + \chi_1\nabla u \cdot \nabla c[v] +  u (\chi_1\|c[v]\|_{L^\infty} + b_1) - b_1 u^2.
\end{split}
\]
To control $\|c[v]\|_{L^\infty}$, note that for any $p\in(n,\infty)$ (e.g. we can fix $p=2n$), the Gagliardo--Nirenberg inequality yields 
\begin{equation}\label{cv_again}
\|c[v]\|_{L^\infty} \leq C \|c[v]\|_{W^{1,p}}^\theta \|c[v]\|_{L^1}^{1-\theta} \leq C \|v\|_{L^p}^\theta \|v\|_{L^1}^{1-\theta}
\end{equation} for $\theta=\frac{np}{np-n+p}$, where the second inequality follows from Lemma~\ref{lemma_c}. As a result, the uniform-in-time bound on $\|v\|_{L^1}$ and $\|v\|_{L^p}$ in Step 1 and 2 yields that $\|c[v(t)]\|_{L^\infty}\leq C$ for all $t\geq 0$, where $C$ depends on $u_0, v_0, n, \Omega$ and the coefficients of \eqref{simplified}.

For such $C$, let $w$ be the constant-in-space solution to the ODE
\begin{equation}\label{w_temp}
\partial_t w = (\chi_1 C + b_1)w - b_1 w^2
\end{equation}
with initial data $w = \|u_0\|_{L^\infty}$. Such $w$ is a supersolution to the PDE for $u$, thus comparison principle yields that $u(\cdot,t)\leq w(t)$ for all time during the existence of $u$. Since $w(t)\leq \max\{w(0), \frac{\chi_1 C + b_1}{b_1}\}$ for all time, $\|u(t)\|_{L^\infty}$ also has the same upper bound. An identical argument can be done for $v$, thus the solution $(u,v)$ exists globally in time. 

Finally, for any $t_0>0$, recall that in step 1 and 2 we have that $\sup_{t>t_0} M(t) \leq g_M(t_0)$ and $\sup_{t>t_0} F_p(t) \leq g_p(t_0)$, where $g_M$ and $g_p$ are independent of the initial data. Applying these bounds to \eqref{cv_again} yields that $\sup_{t>t_0}\|c[v]\|_{L^\infty} \leq C(t_0)$ for some $C(t_0)$ independent of the initial data. Applying this bound to \eqref{w_temp}, we have that for any $t_0>0$, $\sup_{t>t_0}\|u(t)\|_{L^\infty}$ has an upper bound that is independent of $u_0,v_0$ (although the bound goes to infinity as $t_0\to 0^+$), and the same can be done for $v$. This finishes the proof of \eqref{estimate_linfty}.
\end{proof}

\begin{remark}
Once we obtain \eqref{estimate_linfty}, we can apply regularity theory for parabolic equations to upgrade it to the following: for any $t_0>0$, there exists some $\alpha\in (0,1)$, such that 
\begin{equation}\label{C_alpha}
\|u\|_{C^{\alpha, \frac{\alpha}{2}}(\overline{\Omega} \times [t_0,\infty))} + \|v\|_{C^{\alpha, \frac{\alpha}{2}}(\overline{\Omega} \times [t_0,\infty))} \leq h(t_0),
\end{equation}
where $h(t_0)$ only depends on $t_0, \Omega, n$ and the coefficients of \eqref{simplified}. To see this, we rewrite the $u$ equation following the notation of \cite{PV} (where we treat $v$ and $c[v]$ as a priori given functions):
\[
\partial_t u - \nabla\cdot a(u,\nabla u) = b(u),
\]
where 
\[
a(u,\nabla u) = d_1 \nabla u + \chi_1 u \nabla c[v] \quad\text{ and } b(u) = b_1 u(1-u-a_1 v).
\]
Recall that \eqref{estimate_linfty} gives $v \in L^\infty(\overline{\Omega}\times[t_0,\infty)) \leq g(t_0)$ (where $g(t_0)$ is independent of the initial data), and combining this with Lemma~\ref{lemma_c} we also have $\nabla c[v] \in L^\infty(\overline{\Omega}\times[t_0,\infty))$. We can then apply \cite[Theorem 1.3]{PV} to conclude that $\|u\|_{C^{\alpha, \frac{\alpha}{2}}(\overline{\Omega} \times [t_0,\infty))} \leq h(t_0) $ (where $h(t_0)$ is independent of the initial data), the the same can be done for the $v$ equation. 

\end{remark}

\subsection{No extinction for weak competition}
In this subsection we aim to prove Theorem~\ref{no_extinction}, which shows that for two mutually repulsive species, if the coefficients $a_1, a_2$ are in the weak competition regime, i.e. $a_1, a_2 \in (0,1)$,  neither of the species would become extinct as $t\to\infty$ regardless of how large $\chi_1,\chi_2$ is.

\begin{proof}[\textbf{\textup{Proof of Theorem~\ref{no_extinction}}}]
We start with showing a useful fact that will be used often in this proof: for all $t\geq 1$, $\|u(t)\|_{L^1}$ being small implies that $\|u(t)\|_{L^\infty}$ is small; and the same can be said for $v(t)$. Namely, there exist some $\alpha\in(0,1)$ and $C$ that are independent of the initial data, such that 
\begin{equation}\label{Linfty_L1}
\|u(t)\|_{L^\infty} \leq C \|u(t)\|_{L^1}^{\frac{\alpha}{\alpha+n}}  \text{ and } \|v(t)\|_{L^\infty} \leq C \|v(t)\|_{L^1}^{\frac{\alpha}{\alpha+n}}  \quad\text{ for all } t\geq 1.
\end{equation}
To see this, recall that \eqref{C_alpha} yields  
\begin{equation}\label{c1}
\sup_{t\geq 1}\|u(t)\|_{C^\alpha(\overline\Omega)}+ \|v(t)\|_{C^\alpha(\overline\Omega)}\leq C_1,
\end{equation} where $C_1$ and $\alpha$ are independent of the initial data. If for some $t\geq 1$ the maximum of $u(\cdot,t)$ is achieved at $x_0$, we have $u(x,t)\geq \frac{\|u(t)\|_{L^\infty}}{2}$ for all $|x-x_0|\leq (\frac{\|u(t)\|_{L^\infty}}{2C_1})^{1/\alpha}$, thus
\[
\|u(t)\|_{L^1} \geq c(\alpha,n, C_1) \|u(t)\|_{L^\infty}^{1+\frac{n}{\alpha}}.
\]
A similar argument can be done for $v$, which finishes the proof of \eqref{Linfty_L1}.

In the rest of the proof we aim to prove the inequality for $u$ in \eqref{goal_l1}, and a parallel argument yields the inequality for $v$. %By Theorem~\ref{thm_bound}, for all $t>1$,  
%$
%\|u(t)\|_{H^2}, \|v(t)\|_{H^2}, \|u(t)\|_\infty, \|v(t)\|_\infty
%$ are all bounded above by some constant $C_1$, depending only on $\Omega$ and the coefficients of \eqref{simplified}. In particular, $C_1$ is independent of the initial data.
For a sufficiently small $0<\delta \ll 1$ to be fixed later, let $t_1>1$ be the first time such that 
\[
M_u(t_1) = \delta \min\{M_u(1),1\}.
\]
(Recall that $M_u(t):=\int_\Omega u(x,t) dx = \|u(t)\|_{L^1}.$) By definition of $t_1$, we have $M_u(t)> \delta \min\{M_u(1),1\}$ for $t \in (1,t_1)$, which implies that 
\begin{equation}\label{time_der}
\frac{d}{dt}M_u(t_1)\leq 0.
\end{equation}

Our goal is to show that if $0<\delta\ll 1$ is sufficiently small, we must have 
\begin{equation}
\label{eq_goal0}
1-u(\cdot,t_1)-a_1 v(\cdot,t_1) > 0 \quad\text{ in } \Omega.
\end{equation} Once this is shown, it immediately implies that 
$
\frac{d}{dt}M_u(t_1) = b_1 \int_\Omega u(1-u-a_1 v) dx>0 ,
$ contradicting \eqref{time_der}. 

Using that $a_1 \in (0,1)$, throughout the rest of the proof we define 
\[
\gamma_1:= \frac{1-a_1}{4} > 0, \quad\gamma_2:= \frac{1+a_1}{2a_1}>1,\]
 and note that they satisfy $1-\gamma_1 -a_1 \gamma_2 = \frac{1-a_1}{4}> 0$. %(For example, $c_1 = \frac{1-a_1}{4}$ and $c_2 = \frac{1+a_1}{2a_1}$ would work.) 
To show \eqref{eq_goal0}, it suffices to show that $\|u(t_1)\|_{L^\infty} \leq \gamma_1$ and $\|v(t_1)\|_{L^\infty} \leq \gamma_2$ for $\delta\ll 1$. The $u$ estimate immediately follows from \eqref{Linfty_L1}: since $\|u(t_1)\|_{L^\infty} \leq C \|u(t_1)\|_{L^1}^{\frac{\alpha}{\alpha+n}} \leq C \delta^{\frac{\alpha}{\alpha+n}}$, one can simply choose $\delta\ll 1$ such that $\|u(t_1)\|_{L^\infty}\leq \gamma_1$.

The rest of the proof is dedicated to showing that 
\begin{equation}\label{goal_v}
\|v(t_1)\|_{L^\infty} \leq \gamma_2 \quad\text{ if }\delta \ll 1.
\end{equation} This estimate is more involved, since $v$ is coupled to $u$ in a more delicate way.  Note that \eqref{c1} already gives that $\|v(t_1)\|_{L^\infty} \leq C_1$, so it suffices to focus on the case $C_1 > \gamma_2$. 

Heuristically speaking, the idea to prove \eqref{goal_v} is that if $M_u(t_1)\leq \delta\ll 1$, we will show that $u$ must have stayed small for a long time interval $[t_1-T,t_1]$ before time $t_1$. During this interval, $v$ almost does not feel the presence of $u$. Note that if $u\equiv 0$ in the $v$ equation in \eqref{simplified}, $v$ would decay exponentially to 1. Thus if $u\ll 1$ in a long time interval $[t_1-T,t_1]$, we can still show (using comparison principle) that $v$ drops below $\gamma_2>1$ at time $t_1$. 
The proof is divided into the following steps:

\textbf{Step 1. Smallness of $M_u(t)$ before time $t_1$.} In this step, we will show that if $ \delta\ll 1$, $M_u(t)$ must have stayed small for a long time interval before time $t_1$. 
To see this, note that for all $t>1$, $M_u(t)$ can at most decay exponentially: 
\begin{equation}\label{diff_ineq_mv}
\frac{d}{dt} M_u(t) = b_1 \int_\Omega u\underbrace{(1-u-a_1 v)}_{\geq -2C_1} dx \geq -2C_1 b_1 M_u(t) \quad\text{ for all $t\geq 1$,}
\end{equation}
where we used that $a_1<1$ and $u(\cdot,t), v(\cdot,t) < C_1$ for all $t\geq 1$. This differential inequality yields that
\begin{equation}\label{temp_mu}
M_u(t) \leq e^{2C_1 b_1 (t_1-t)} M_u(t_1)% \leq e^{2C_1 b_1 (t_1-t)} \delta %\min\{M_u(1), 1\} 
\quad\text{ for all } 1\leq t \leq t_1.
\end{equation}
As a result, for some fixed  $\tilde\delta\in(0,1)$ and $T>0$ (they only depend on $\Omega, n$ and the coefficients of \eqref{simplified}, and will be fixed momentarily in Step 2 and 4),  if $\delta$ is sufficiently small such that $e^{2C_1 b_1 T} \delta \leq \tilde\delta$, then \eqref{temp_mu} implies that for all $\max\{1,t_1-T\} \leq t \leq t_1$,
\begin{equation}\label{M_u_temp}
\begin{split}
M_u(t) &\leq e^{2C_1 b_1 T}  M_u(t_1) = e^{2C_1 b_1 T}  \delta \min\{M_u(1),1\}\\
& \leq \tilde\delta \min\{M_u(1),1\}.
\end{split}
\end{equation}
Note that in this case we automatically have $t_1-T>1$: if not, we would have $M_u(1)\leq \tilde\delta M_u(1)$, contradicting that $\tilde\delta\in (0,1)$. Thus \eqref{M_u_temp} actually holds for $t\in[t_1-T, 1]$.

\textbf{Step 2. Choice of $\tilde\delta$.} In this step, we fix $\tilde\delta \in (0,1)$ such that \eqref{M_u_temp} implies \begin{equation}\label{c_infty_bd}
\|c[u(t)]\|_{L^\infty}\leq \frac{b_2(\gamma_2-1)}{2\chi_2}\quad \text{ for }t\in[t_1-T, t_1].
\end{equation}
 To control $\|c[u(t)]\|_{L^\infty}$, recall that 
for any $p\in(n,\infty)$ (e.g. fix $p=2n$), the Gagliardo--Nirenberg inequality and Lemma~\ref{lemma_c} yield 
\begin{equation}
\|c[u]\|_{L^\infty} \leq C \|c[u]\|_{W^{1,p}}^\theta \|c[u]\|_{L^1}^{1-\theta} \leq C \|u\|_{L^p}^\theta \|u\|_{L^1}^{1-\theta} \leq C \|u\|_{L^\infty}^{\theta(1-\frac{1}{p})} \|u\|_{L^1}^{1-\theta(1-\frac{1}{p})}
\end{equation} for $\theta=\frac{np}{np-n+p}$, where the second inequality follows from Lemma~\ref{lemma_c} and the third inequality follows from H\"older's inequality. 
Combining this with \eqref{Linfty_L1} gives that
\begin{equation}\label{c_temp}
\|c[u(t)]\|_{L^\infty} \leq C\|u(t)\|_{L^1}^{\beta} \quad\text{ for all }t\geq 1
\end{equation}
for some $\beta>0$. Thus if $\tilde\delta$ is sufficiently small, we obtain  \eqref{c_infty_bd} by combining \eqref{M_u_temp} and \eqref{c_temp}.

\textbf{Step 3. Controlling $v$ by comparison principle}. During the time interval $[t_1-T, t_1]$, using that $\Delta c[u] = c[u]-u$ and $u,v\geq 0$, the equation for $v$ satisfies
\[
\begin{split}
\partial_t v &= d_2 \Delta v + \chi_2\nabla v \cdot \nabla c[u] + \chi_2 v (c[u]-u) + b_2 v(1-v-a_2 u)\\
&\leq d_2\Delta v + \chi_2\nabla v \cdot \nabla c[u] +  b_2 v \left(1 + \frac{\chi_2}{b_2}\|c[u]\|_{L^\infty} - v \right) .
\end{split}
\]
Combining this with \eqref{c_infty_bd}, we have
\begin{equation}\label{v_ineq}
\partial_t v \leq d_2\Delta v + \chi_2\nabla v \cdot \nabla c[u] +  b_2 v \left(\frac{1+\gamma_2}{2} - v \right) \quad\text{ for all }t\in[t_1-T,t_1],
\end{equation}
and since $t_1-T>1$, we have $\|v(t_1-T)\|_{L^\infty} \leq C_1$ due to \eqref{c1}.

To control $v$ from above, we will compare it with $w$, which solves the ODE
\begin{equation}\label{w_temp2}
 w' = b_2 w \left(\frac{1+\gamma_2}{2} - w \right)
\end{equation}
with initial data $w(t_1-T) = C_1$. Comparing it with \eqref{v_ineq}, we know that  $w$ is a supersolution to the PDE for $v$, thus comparison principle yields that 
\begin{equation}\label{vw}
v(\cdot,t)\leq w(t)\quad\text{ for all }t\in[t_1-T, t_1].
\end{equation}

\textbf{Step 4. Choice of $T$.} In this step, we choose $T$ such that the solution $w$ to \eqref{w_temp}  with initial data $w(t_1-T) = C_1$ satisfies $w(t_1)\leq \gamma_2$. Recall that in the discussion after \eqref{goal_v}, it suffices to focus on the case $C_1>\gamma_2$, and note that it implies $C_1>\frac{1+\gamma_2}{2}$ since $\gamma_2>1$. In this case $w(t)$ is monotone decreasing in time. In addition, for any $t>t_1-T$ such that $w(t)>\gamma_2$, \eqref{w_temp2} implies that $w'(t) \leq b_2 \gamma_2 \frac{1-\gamma_2}{2}<0$. As a result, we can set $T := \frac{2(C_1-\gamma_2)}{b_2 \gamma_2 (1-\gamma_2)}$. Such $T$ guarantees that $\omega(t_1)\leq \gamma_2$: if not, we would have $\omega(t)>\gamma_2$ in $[t_1-T,t_1]$, thus $\omega(t_1) \leq C_1 + Tb_2 \gamma_2 \frac{1-\gamma_2}{2} = \gamma_2$, a contradiction. 

Finally, combining such choice of $T$ with step 3, \eqref{vw} implies that
\[
v(\cdot,t_1) \leq w(t_1) \leq \gamma_2,
\]
finishing the proof of \eqref{goal_v}.
\end{proof}

\section{Stability and dynamics in the weak competition regime}\label{sec3}

In this section, we will derive the criteria under which the constant steady state \eqref{steady_general} is linearly stable, and obtain the approximate dynamics of the system when \eqref{steady_general} is ``slightly linearly unstable''. Since the system \eqref{simplified}--\eqref{poisson} has many parameters, to make the analysis easier, let us consider the system where the coefficients for $u$ and $v$ are symmetric about each other. Let us set $d_1=d_2=1$, $b_1=b_2=1$, $a_1=a_2=a$, and $\chi_1=\chi_2=\chi>0$. That is, \eqref{simplified} now becomes
\begin{equation}\label{simplified2}
\begin{cases}
\partial_t u  = \Delta u +   \chi \nabla\cdot( u \nabla c[v]) +  u(1 - u - a v) &\text{ in }\Omega \times (0,T), \\
\partial_t v =  \Delta v +   \chi \nabla\cdot( v \nabla c[u]) +  v(1 - v - a u) &\text{ in }\Omega \times (0,T),  \\
\frac{\partial u}{\partial \nu} = \frac{\partial v}{\partial \nu} = 0 & \text{ on } \partial \Omega  \times (0,T),
\end{cases}
\end{equation}
where $c[u]$ and $c[v]$ again solve the elliptic equation \eqref{poisson}. Throughout this section, we assume $a\in(0,1)$ is a fixed parameter in the weakly competition regime.

\subsection{Linear stability analysis for a general domain}\label{linear_stability}
For the system \eqref{simplified2}--\eqref{poisson}, one can easily check that \eqref{steady_general} is still the unique positive steady state, and since we have set that $a_1=a_2=a$, it can be written as
\begin{align}\label{Eq: constant steady state}
\begin{split}
\begin{pmatrix}
\overline{u}\\ 
\overline{v}
\end{pmatrix}\ =\ \frac{1}{a + 1}
\begin{pmatrix}
1\\
1
\end{pmatrix},
\end{split}
\end{align}
whose corresponding chemical concentrations are $c[\overline{ u}]=c[\overline{v}] = \frac{1}{a + 1}$. For a fixed $a\in(0,1)$, in this section we will perform the standard linear stability analysis to determine the linear stability criteria of \eqref{linear_stability} in terms of $\chi$.

Since $\Omega \subset \R^n$ is a bounded domain with smooth boundary, it is well known that the Neumann eigenvalue problem
\begin{equation}\label{eigenvalue}
\begin{cases}
\Delta \varphi + \lambda \varphi = 0 &\text{ in } \Omega\\
\frac{\partial \varphi}{\partial \nu}  = 0 &\text{ on } \partial \Omega
\end{cases}
\end{equation}
has an infinite sequence of eigenfunctions $\{\varphi_k\}_{k=0}^\infty$ with corresponding eigenvalues $0=\lambda_0 < \lambda_1 \leq \dots \leq \lambda_k \leq \dots$, where $\{\varphi_k\}_{k=0}^\infty$ form an orthonormal basis of $L^2(\Omega)$, and $\varphi_0 = |\Omega|^{-1/2}$.  In particular, in the special case $n=1$ and $\Omega=[0,L]$, we have 
\begin{equation}\label{interval}
\varphi_k(x) = \sqrt{2}|\Omega|^{-1/2} \cos(w_k x) ~~\text{ and }~~ \lambda_k = w_k^2 \quad\text{ for }k \in \mathbb{N}^+,
\end{equation} where $w_k := \frac{k\pi}{L}$.

For a solution $(u,v)$ to \eqref{simplified2}--\eqref{poisson}, we can use the above Neumann eigenfunction expansion to write it as
\[
u(x,t) = \overline{u} +  \sum_{k = 0}^{\infty}a_{k}(t)\varphi_k(x),\quad
v(x,t) = \overline{v} +  \sum_{k = 0}^{\infty}b_{k}(t)\varphi_k(x).
\] 
When $(u,v)$ has a small perturbation to the constant steady state \eqref{Eq: constant steady state} (i.e. all $|a_k|,|b_k|\ll 1$ for $k\in N$),
substituting $u(x,t)$ and $v(x,t)$ into \eqref{simplified2}--\eqref{poisson}, for each mode $k\in \mathbb{N}$, $a_k$ and $b_k$ satisfy the ODE
\[
\frac{d}{dt} \begin{pmatrix}a_k(t)\\ b_k(t)  \end{pmatrix} =\mathcal{J}_k  \begin{pmatrix}a_k(t)\\ b_k(t)  \end{pmatrix}+ O(\|u-\overline u\|^2 +\|v-\overline v\|^2),
\]
where $\mathcal{J}_k$ is a $2\times 2$ matrix given by
\begin{equation}\label{def_Jk}
\mathcal{J}_k := \begin{pmatrix} -\lambda_k - \overline{u} & -\chi \overline{u}\frac{ \lambda_k}{1+\lambda_k} - a\overline{u} \\
  -\chi \overline{v}\frac{ \lambda_k}{1+\lambda_k} - a\overline{v}  &  -\lambda_k - \overline{v} 
 \end{pmatrix}\quad\text{ for }k\in\mathbb{N}.
 \end{equation}
The linear stability of the constant steady state $(\overline{u},\overline{v})$ is determined by the eigenvalues of the matrices $\mathcal{J}_k$: namely, $(\overline{u},\overline{v})$ is linearly stable if and only if all eigenvalues of $\mathcal{J}_k$ have negative real parts for all $k\in\mathbb{N}$.

 Due to our assumptions $\chi>0, a> 0$ and the fact that $\overline u = \overline v>0$, each $\mathcal{J}_{k}$ is of the form $\begin{pmatrix} c_k & d_k \\ d_k & c_k\end{pmatrix}$ with $c_k<0$ and $d_k<0$, thus its eigenvalues are $c_k\pm d_k$. Clearly we have $c_k + d_k<0$ for all $k\in \mathbb{N}$. For $c_k-d_k$, we have
 \[
 c_k-d_k = \frac{\chi \lambda_k}{(a+1)(1+\lambda_k)} - \lambda_k + \frac{a-1}{a+1}\quad\text{ for all } k\in \mathbb{N},
 \]
thus $c_0-d_0<0$ (since $\lambda_0=0$), and for $k\in\mathbb{N}^+$ we have $c_k-d_k<0$ if and only if $0<\chi< \chi_k^*$, where
\begin{equation}\label{def_chi_k}
\chi_k^* := \frac{1+\lambda_k}{\lambda_k}(\lambda_k(a+1) + 1-a) = 2 + \frac{1-a}{\lambda_k} + \lambda_k(1+a)\quad\text{ for } k\in \mathbb{N}^+.
\end{equation}
Note that $\chi_k^*>0$ for each $k\in \mathbb{N}^+$, and $\chi_k^* \to \infty$ as $k\to\infty$ (since $\lim_{k\to\infty} \lambda_k = \infty$).
This leads to the following criteria for the linear stability of constant steady state:
\begin{proposition}\label{remark: 1D linear critical}
For the system \eqref{simplified2}--\eqref{poisson} with $\chi>0$, $0<a<1$, the positive constant solution \eqref{Eq: constant steady state} is linearly stable  if $0<\chi < \chi^{*}$, where 
$
\chi^* := \inf_{k\in\mathbb{N}^+}\chi_k^*,
$ with $\chi_k^*$ given by \eqref{def_chi_k}.
\end{proposition}

\subsection{Weakly nonlinear analysis for a 1D interval}\label{sec_weakly_nonlinear}
In this subsection, we focus on the case $n=1$ and $\Omega=[0,L]$, and consider the regime $\chi > \chi^*$, where the positive constant steady state \eqref{Eq: constant steady state} is linearly unstable. The goal is to perform weakly nonlinear analysis in the regime $0<\chi-\chi^* \ll 1$, and derive the amplitude equations for system \eqref{simplified2}--\eqref{poisson} using a standard perturbation expansion approach, following \cite[Section 6]{cross2009pattern}.

Let us point out that when $\Omega=[0,L]$, \eqref{interval} gives that $\lambda_k = w_k^2=(\frac{k\pi}{L})^2$, thus 
\[
\chi_k^* = 2 + \frac{1-a}{w_k^2} + (1+a)w_k^2= 2 + \frac{(1-a)L^2}{k^2 \pi^2} + (1+a)\frac{k^2 \pi^2}{L^2}
\] is a convex function in $k$ (for $k\in\mathbb{R}^+$). As a result, in the definition $\chi^* := \inf_{k\in\mathbb{N}^+}\chi_k^*$, the infimum is achieved at either one or two modes. 

In the rest of this section, without loss of generality we assume that the infimum is achieved at only one mode (denote it by $k^* \in \mathbb{N}^+$), which is the case for almost every  $L > 0$. In this case, we have $\chi^* = \chi_{k^*}^*$, and $\chi^* < \chi_k^*$ for all $k\in \mathbb{N}^+ \setminus\{k^*\}$. In particular, when $\chi=\chi^*$, the matrices $\mathcal{J}_k$ defined in \eqref{def_Jk} (call them $\mathcal{J}_k^*$) are invertible for any $k=\N\setminus \{k^*\}$ and have two negative eigenvalues, and is singular only at $k=k^*$ with $(1,-1)^T$ in its kernel.

In the weakly nonlinear regime $\chi = \chi^* + \varepsilon$ with $0<\varepsilon\ll 1$,  we have $\chi>\chi_{k^*}^*$, and $\chi<\chi_k^*$ for all $k\in \mathbb{N}^+ \setminus\{k^*\}$. Therefore, among all matrices $\mathcal{J}_k$ defined in \eqref{def_Jk}, only $\mathcal{J}_{k^{*}}$ has a positive eigenvalue, given by
\begin{align}\label{Eq: disper}
\begin{split}
\sigma(k^*) = \varepsilon \frac{\overline{u}w_{k^*}^{2}}{1 + w_{k^*}^{2}},
\end{split}
\end{align}
and its corresponding eigenvalue is $(1,-1)^T$. It means for a small perturbation near the constant steady state \eqref{Eq: constant steady state}, only the mode $\cos(w_{k^*}x)(1,-1)^T$ is growing exponentially with growth rate $\sigma(k^*)$ that is of order $\varepsilon$, and all the other modes are decaying. 
Our goal is to derive the \emph{amplitude equation} that gives the dynamics for this growing mode.

Since this mode has a tiny growth rate of order $\varepsilon$, we are motivated to define a slow time variable $\tau := \varepsilon t$ that describes the dynamics of the system when $\chi=\chi^*+\varepsilon$ is slightly above the stability threshold. Under the slow time variable, the dynamics in \eqref{simplified2} can be written as
\begin{equation}\label{simplified3}
\begin{cases}
\varepsilon \partial_\tau u  = \partial_{xx} u +   (\chi^*+\varepsilon) \,\partial_x ( u\, \partial_x  c[v]) +  u(1 - u - a v) &\text{ in }[0,L] \times (0,\infty), \\
\varepsilon \partial_\tau v =  \partial_{xx} v +    (\chi^*+\varepsilon)\, \partial_x ( v\, \partial_x  c[u]) +  v(1 - v - a u) &\text{ in }[0,L] \times (0,\infty),  \\
\partial_x u = \partial_x v = 0 & \text{ on }  \{0,L\}  \times (0,\infty).
\end{cases}
\end{equation}

Next we will express $u$ and $v$ as the following expansions in terms of the small parameter $0<\varepsilon \ll 1$:
\begin{align}\label{Eq: expansion}
\begin{split}
\begin{pmatrix}
u(x,\tau)\\ 
v(x,\tau)
\end{pmatrix} = \begin{pmatrix}
\overline{u}\\ 
\overline{v}
\end{pmatrix} + \sum_{j = 1}^{\infty} \varepsilon^{\frac{j}{2}}\begin{pmatrix}
u^{(j)}(x,\tau)\\ 
v^{(j)}(x,\tau)
\end{pmatrix},
\end{split}
\end{align}
where we chose the expansion to be powers of $\varepsilon^{1/2}$, since we expect a pitchfork bifurcation. To see this, using the symmetry between the $u,v$ equation and the fact that the only unstable mode for $0<\varepsilon\ll 1$ is of the form $\cos(w_{k^*}x)(1,-1)^T$, we expect that for small $\varepsilon$ there are two steady states approximately of the form $(\overline u + f, \overline v - f)$ and $(\overline u-f, \overline v+f)$ for some small $f$, suggesting a pitchfork bifurcation.

Again, the Neumann boundary condition in \eqref{simplified2} allows us to expand each $(u^{(j)}(x,T),
v^{(j)}(x,T))^T$ in cosine series:
\begin{align}\label{expansion_j}
\begin{split}
\begin{pmatrix}
u^{(j)}(\tau,x)\\v^{(j)}(\tau,x)
\end{pmatrix} = \sum_{k=0}^\infty \cos(w_k x)\begin{pmatrix}
 a_k^{(j)}(\tau)\\   b_k^{(j)}(\tau)
\end{pmatrix} \quad\text{ for } j\in \mathbb{N}^+.
\end{split}
\end{align}
Once $(u,v)$ is written in the expansion \eqref{Eq: expansion} in terms of $\varepsilon$, their chemicals $(c[u], c[v])^T$ can also be written in an expansion in terms of $\varepsilon$ as 
\begin{align}\label{expansion_j_c0}
\begin{split}
\begin{pmatrix}
c[u(\tau)]\\c[v(\tau)]
\end{pmatrix}= \begin{pmatrix}
\overline{u}\\ 
\overline{v}
\end{pmatrix} + \sum_{j = 1}^{\infty} \varepsilon^{\frac{j}{2}}\begin{pmatrix}
c[u^{(j)}]\\ 
c[v^{(j)}]
\end{pmatrix},
\end{split}
\end{align}
 where for each $j$, using \eqref{poisson} and \eqref{expansion_j}, the term $(c[u^{(j)}], c[v^{(j)}])^T$ is given by
\begin{align}\label{expansion_j_c}
\begin{split}
\begin{pmatrix}
c[u^{(j)}(\tau)]\\c[v^{(j)}(\tau)]
\end{pmatrix} = \sum_{k=0}^\infty \frac{1}{1+w_k^2}\cos(w_k x)\begin{pmatrix}
 a_k^{(j)}(\tau)\\   b_k^{(j)}(\tau)
\end{pmatrix} \quad\text{ for } j\in \mathbb{N}^+.
\end{split}
\end{align}

Next we will plug the expansions \eqref{Eq: expansion} and \eqref{expansion_j_c0} into \eqref{simplified3}, and collect terms of order $\varepsilon^{\frac{j}{2}}$ for each $j \in \N$ respectively. Note that there are no terms of order one in \eqref{simplified3}, and the lowest order terms are of order $\varepsilon^{\frac12}$.

\noindent $\bullet$ \textbf{Order $\varepsilon^{\frac12}$ terms: }
Collecting all terms of order $\varepsilon^{\frac12}$, we have
\begin{align}\label{Eq: order01}
\begin{split}
\begin{pmatrix}
0\\ 
0
\end{pmatrix} =
\begin{pmatrix}
\partial_{xx} u^{(1)} + \chi^* \overline{u} \,\partial_{xx} c[v^{(1)}] - \overline{u}(u^{(1)} + a v^{(1)})\\ 
\partial_{xx} v^{(1)}  + \chi^* \overline{v} \,\partial_{xx} c[u^{(1)}] - \overline{v}(v^{(1)} + a u^{(1)})
\end{pmatrix}. 
\end{split}
\end{align}

Plugging \eqref{expansion_j} and \eqref{expansion_j_c} with $j=1$ into \eqref{Eq: order01}, the equation becomes 
\begin{equation}\label{expansion_temp}
\sum_{k=0}^\infty  \cos(w_k x)  \mathcal{J}_k^*
\begin{pmatrix}
 a_k^{(1)}(\tau)\\   b_k^{(1)}(\tau)
\end{pmatrix}= 
\begin{pmatrix}
0 \\ 0
\end{pmatrix},
\end{equation}
where $\mathcal{J}_k^*$ is the matrix $\mathcal{J}_k$ in \eqref{def_Jk} with $\chi$ replaced by $\chi^*$. As we discussed in the beginning of this subsection, $\mathcal{J}_k^*$ is invertible for all $k\in \mathbb{N}$ except at $k=k^*$, where $\mathcal{J}_{k^*}^*$ has an eigenvalue 0 with corresponding eigenvector $(1, -1)^T$. Combining this fact with the orthogonality of $\{\cos(w_k x)\}_{k\in\mathbb{N}}$, we know that $(a_k^{(1)}(\tau),b_k^{(1)}(\tau))^T\equiv 0$ for all $k\neq k^*$, and for $k=k^*$ we have $(a_{k^*}^{(1)}(\tau),b_{k^*}^{(1)}(\tau))^T$ is a multiple of $(1,-1)^T$ for each $\tau$. Thus $(u^{(1)}, v^{(1)})^T$ must take the form
\begin{align}\label{Eq: solution in order01}
\begin{split}
\begin{pmatrix}
u^{(1)}\\v^{(1)}
\end{pmatrix} = P(\tau) \cos(w_{k^*}x)\begin{pmatrix}
1\\ -1
\end{pmatrix}
\end{split}
\end{align}
where the amplitude $P(\tau)$ can be any function of $\tau$. (Note that this implies that $c[u^{(1)}]=\frac{P(\tau)}{1+w_{k^*}^2} \cos(w_{k^*}x) $, with $c[v^{(1)}]$ being its negative.) As we explained above, our goal is to find the amplitude equation, which is the differential equation that $P(\tau)$ satisfies. In order to do this, we need to go to higher order terms.

\noindent $\bullet$ \textbf{Order $\varepsilon$ terms: }
The equations of order $\varepsilon$ satisfy the following, and note that the right hand side contain quadratic products of the order $\varepsilon^{1/2}$ terms $u^{(1)}, v^{(1)}$:
\begin{align}\label{Eq: order02}
\begin{split}
\begin{pmatrix}
0\\ 
0
\end{pmatrix} = \begin{pmatrix}
\partial_{xx} u^{(2)} + \chi^* \overline{u} \partial_{xx} c[v^{(2)}] - \overline{u}(u^{(2)} + a v^{(2)})\\ 
\partial_{xx} v^{(2)} + \chi^* \overline{v} \partial_{xx} c[u^{(2)}] - \overline{v}(v^{(2)} + a u^{(2)})
\end{pmatrix} + 
\begin{pmatrix}
\chi^* \partial_x (u^{(1)} \partial_x c[v^{(1)}])- u^{(1)} (u^{(1)}+ a v^{(1)})\\ 
\chi^* \partial_x (v^{(1)} \partial_x c[u^{(1)}])- v^{(1)} (v^{(1)}+ a u^{(1)})\end{pmatrix}.
\end{split}
\end{align}
As before, we plug \eqref{expansion_j} and \eqref{expansion_j_c} (with $j=1,2$) into \eqref{Eq: order02}. Note that the first term on the right hand side is identical as the right hand side of \eqref{Eq: order01}, except that all $1$'s in the superscript become $2$'s. Thus it is equal to the left hand side of \eqref{expansion_temp}, except that $a_k^{(1)}$ and $b_k^{(1)}$ are replaced by $a_k^{(2)}$ and $b_k^{(2)}$. To deal with the second term on the right hand side of \eqref{Eq: order02}, we first plug in the expressions of $u^{(1)}$ and $v^{(1)}$ from \eqref{Eq: solution in order01}, then use the trigonometric identities  $\cos(w_{k^*}x)^2 =  \frac{1}{2} + \frac{1}{2} \cos(2w_{k^*} x)$ and $\sin(w_{k^*} x) \cos(w_{k^*}x) = \frac{1}{2} \sin(2w_{k^*}x )$ to write the quadratic products into a linear combination of trigonometric functions. After an elementary computation (which is omitted here), \eqref{Eq: order02} becomes
\begin{align}\label{Eq: order02_temp}
\begin{split}
\begin{pmatrix}
0\\ 
0
\end{pmatrix} =   \sum_{k=0}^{\infty} \cos(w_k x) \mathcal{J}_k^* 
\begin{pmatrix}
 a_k^{(2)}(\tau)\\   b_k^{(2)}(\tau)
\end{pmatrix} + 
P(\tau)^2\left(\frac{a - 1}{2} + c_{1}\cos(2w_{k^*}x) \right)
\begin{pmatrix}
1\\ 
1
\end{pmatrix},
\end{split}
\end{align}
where
\begin{align}\label{Eq: order02 supp}
\begin{split}
c_{1} =  \chi^*\frac{w_{k^*}^{2}}{1 + w_{k^*}^{2}} + \frac{a - 1}{2} = \frac{1-a}{2} + (1+a) w_{k^*}^2 > 0,
\end{split}
\end{align}
and the second identity follows from \eqref{def_chi_k}, which gives  $\chi^*\frac{\lambda_{k^*}}{1 + \lambda_{k^*}}=\lambda_{k^*}(a+1)-(a-1)$ (recall that $\lambda_{k^*}=w_{k^*}^2$). 

Again using orthogonality of $\{\cos(w_k x)\}_{k=0}^\infty$ and the fact that $\mathcal{J}_k^*$ is invertible for $k\neq k^*$, we know that the coefficients $(a_k^{(2)}(\tau), b_k^{(2)}(\tau))\equiv 0$ for all $k$ except three modes: $k=k^*$, $k=0$, and $k=2k^*$. For the  $k=k^*$ term, we have $(a_{k^*}^{(2)}(\tau), b_{k^*}^{(2)}(\tau))^T = Q(\tau) (1,-1)^T$ for some function $Q(\tau)$, where we do not have any more information about $Q(\tau)$. For the $k=0$ and $k=2k^*$ terms, note that \eqref{Eq: order02_temp} leads to the two equations
\[
  \mathcal{J}_{0}^* \begin{pmatrix}
 a_{0}^{(2)}(\tau)\\   b_{0}^{(2)}(\tau) \end{pmatrix} = -\frac{a-1}{2} P(\tau)^2
 \begin{pmatrix}
1\\ 
1
\end{pmatrix} \quad\text{ and } \quad \mathcal{J}_{2k^*}^* \begin{pmatrix}
 a_{2k^*}^{(2)}(\tau)\\   b_{2k^*}^{(2)}(\tau) \end{pmatrix} =- c_1 P(\tau)^2
 \begin{pmatrix}
1\\ 
1
\end{pmatrix}.
\]
Solving the two linear systems of equations using the definitions of $\mathcal{J}_k^*$ in \eqref{def_Jk} with $\chi=\chi^*$ (note that the matrices are easy to invert since they are of the form $\begin{pmatrix} c & d\\ d & c\end{pmatrix}$), we obtain that the order $\varepsilon$ terms are of the form
\begin{align}\label{Eq: order02 solution}
\begin{split}
\begin{pmatrix}
u^{(2)}\\v^{(2)}
\end{pmatrix} = 
P(\tau)^2\left( c_{2}  \cos(2w_{k^*}x) + \frac{a - 1}{2} \right)
\begin{pmatrix}
1\\ 1
\end{pmatrix} + Q(\tau) \cos(w_{k^*}x) \begin{pmatrix}
1\\ -1
\end{pmatrix}
\end{split}
\end{align}
where
\begin{align}\label{Eq: order02 solution supp}
\begin{split}
c_{2} = \frac{c_1}{\chi^*\frac{\overline{u}w_{2k^*}^{2}}{1 + w_{2k^*}^{2}} + w_{2k^*}^{2} + 1} > 0.
\end{split}
\end{align}

\noindent $\bullet$ \textbf{Order $\varepsilon^{\frac32}$ terms: }
Next we collect terms in higher order $\varepsilon^{\frac{3}{2}}$. As we will see, at this order $\frac{d}{d\tau}P(\tau)$ finally appears, which allows us to close the equation for $P(\tau)$. Similar to the computation of the order $\varepsilon$ terms \eqref{Eq: order02}, for order $\varepsilon^{\frac32}$ terms we have
\begin{align}\label{Eq: order03}
\begin{split}
\begin{pmatrix}
\partial_\tau u^{(1)}\\ 
\partial_\tau v^{(1)}
\end{pmatrix} =& \begin{pmatrix}
\partial_{xx} u^{(3)} + \chi^* \overline{u} \partial_{xx} c[v^{(3)}] - \overline{u}(u^{(3)} + a v^{(3)})\\ 
\partial_{xx} v^{(3)} + \chi^* \overline{v} \partial_{xx} c[u^{(3)}] - \overline{v}(v^{(3)} + a u^{(3)})
\end{pmatrix} + 
\begin{pmatrix}
\chi^* \partial_x (u^{(1)} \partial_x c[v^{(2)}])- u^{(1)} (u^{(2)}+ a v^{(2)})\\ 
\chi^* \partial_x (v^{(1)} \partial_x c[u^{(2)}])- v^{(1)} (v^{(2)}+ a u^{(2)})\end{pmatrix}\\
&+ \begin{pmatrix}
\chi^* \partial_x (u^{(2)} \partial_x c[v^{(1)}])- u^{(2)} (u^{(1)}+ a v^{(1)})\\ 
\chi^* \partial_x (v^{(2)} \partial_x c[u^{(1)}])- v^{(2)} (v^{(1)}+ a u^{(1)})\end{pmatrix} + \begin{pmatrix} \bar u \,\partial_{xx} c[v^{(1)}]\\  \bar v \,\partial_{xx} c[u^{(1)}] \end{pmatrix},
\end{split}
\end{align}
where the last term is contributed by the $\varepsilon$ term in $\chi=\chi^*+\varepsilon$.

% For convenience, we focus on the study of $u^{(3)}(x,\tau)$, since the $v$ equation can be done in a parallel way. As we collect the terms of order $\varepsilon^{\frac{3}{2}}$ in \eqref{simplified3} (note that the leading order on the left hand side is of this order), 
Next we plug the expressions for the lower order terms \eqref{Eq: solution in order01} and \eqref{Eq: order02 solution} into above, and again we use the trigonometric identities to rewrite all quadratic products into linear terms. This computation is elementary but somewhat tedious, and we omit it here. Once all terms in \eqref{Eq: order03} are written into a linear combination of cosine series, let us only collect the $\cos(w_{k^*} x)$ terms (which is allowed due to the orthogonality of the cosine series), resulting in the equation
 \begin{align}\label{Eq: order03_temp}
\begin{split}
P'(\tau)
\begin{pmatrix}
1\\-1
\end{pmatrix} = \mathcal{J}_{k^*}^* 
\begin{pmatrix}
 a_{k^*}^{(3)}(\tau)\\   b_{k^*}^{(3)}(\tau)
\end{pmatrix} + \lambda_1 P(\tau)^3\begin{pmatrix}
1\\ -1
\end{pmatrix}+ \lambda_2 P(\tau) \begin{pmatrix}
1\\ -1
\end{pmatrix},
\end{split}
\end{align}
%
%
% we have
%\begin{align}\label{Eq: solution of order03}
%\begin{split}
%\partial_{\tau}u^{(1)} =\, & \partial_{xx} u^{(3)} + \chi^* \overline{u}\,\partial_{xx} c[v^{(3)}] - \overline{u}(u^{(3)} + a v^{(3)}) + \cos(w_{k^*}x)\left(\lambda_{1} P^{3}(T) + \lambda_{2} P(T)\right)\\
%& +  \cos(w_{3k^*}x) g(P(T);\chi^*,k^*),
%\end{split}
%\end{align}
where
\begin{align}\label{Eq: solution of order03 supp}
\begin{split}
\lambda_{1} := \chi^*\left( \frac{(a-1-c_2)w_{k^*}^{2}}{2(1 + w_{k^*}^{2})} - \frac{c_{2} w_{k^*} w_{2k^*}}{2(1 + w_{2k^*}^{2})}\right) - (c_{2} + a - 1);\quad \lambda_{2} := \frac{\overline{u}w_{k^*}^{2}}{1 + w_{k^*}^{2}}>0.
\end{split}
\end{align}
Multiplying $(1,-1)$ on the left to both sides of \eqref{Eq: order03_temp} (and recall that $(1,-1)$ is in the left kernel of $\mathcal{J}_{k^*}^*$, we arrive at a scalar differential equation for $P(\tau)$:%
%Here we omit the exact form for $g(P(T);\chi^*,k^*)$, since it does not contain helpful information for derivation of amplitude equation. From the solvability condition of Fredholm's alternative, we derive the dynamical equation for $P(T)$ such that we have expected non-zero solution $u^{(1)}(x,T)$, $P(T)$ satisfies
\begin{align}\label{Eq: order03 solution}
\begin{split}
\frac{d}{d\tau}P(\tau) =  \lambda_{1}P(\tau)^3 + \lambda_{2}P(\tau),
\end{split}
\end{align} 
with $\lambda_1,\lambda_2$ are given by \eqref{Eq: solution of order03 supp}. 
In order to recover the evolution of the amplitude in the regular time scale $t$, we define $A(t) = \varepsilon^{\frac{1}{2}}P(t)$, which satisfies the amplitude equation
\begin{align}\label{Eq: AMP}
 \begin{split}
 \frac{d}{dt}A(t) =  \varepsilon \lambda_2 A(t) + \lambda_1 A(t)^3.
 \end{split}
\end{align}
%where $\Lambda_{1} = -\lambda_{1}$ and $\Lambda_{2} = \varepsilon \lambda_{2}$. 

As expected, this differential equation undergoes a pitchfork bifurcation at $\varepsilon=0$, %(i.e. at $\chi=\chi^*$)
 where the type of bifurcation (supercritical v.s. subcritical) is determined by the sign of $\lambda_1$ and $\lambda_2$. Note that $\lambda_{2}$ is always positive, whereas the sign of $\lambda_1$ depends on $a$ and $L$ in a rather implicit manner. Using the formula of $\lambda_1$ in \eqref{Eq: solution of order03 supp}, we have numerically checked that $\lambda_1<0$ for all $a\in[0, 1]$ and $L \in [0, 4]$, indicating that \eqref{Eq: AMP} undergoes a supercritical bifurcation in $\varepsilon$ for these parameters. 

Finally, to obtain the dynamical solution $(u(x,t), v(x,t))$ for $0<\varepsilon\ll 1$,  plugging \eqref{Eq: solution in order01} into \eqref{Eq: expansion} (and using the fact that $A(t) = \varepsilon^{\frac{1}{2}}P(t)$), the dynamical solution near the constant steady state can be approximated as follows:
%our solution for $P(t)$ is only valid to the order $O(\varepsilon^{-1})$, so our solution for $u(x, t)$ and $v(x,t)$ are also only valid to this order, and is given by
\begin{align}\label{Eq: multiscale solution}
\begin{split}
\begin{pmatrix}
u(x,t)\\ 
v(x,t)
\end{pmatrix} = \begin{pmatrix}
\overline{u}\\ 
\overline{v}
\end{pmatrix} + A(t)\cos(w_{k^*} x)\begin{pmatrix}
1\\ 
-1
\end{pmatrix} + O(\varepsilon),
\end{split}
\end{align}
where $A(t)$ solves \eqref{Eq: AMP}.

For $\varepsilon>0$, when $\lambda_1<0, \lambda_2>0$, one can easily check that \eqref{Eq: AMP} has three steady states: an unstable steady state zero, and two other stable steady states $\pm \sqrt{\frac{\varepsilon\lambda_{2}}{-\lambda_{1}}}$. Thus for $0<\varepsilon\ll 1$, in addition to the unstable constant steady state $(\overline u, \overline v)$ in \eqref{Eq: constant steady state}, we also expect the system \eqref{simplified2}--\eqref{poisson} to have two stable almost-constant steady states of the form
\begin{align}\label{Eq: stat_sol}
\begin{split}
\begin{pmatrix}
 u_s\\  v_s
\end{pmatrix} = \begin{pmatrix}
\overline u\\ \overline v
\end{pmatrix} \pm \varepsilon^{1/2} \sqrt{\frac{\lambda_{2}}{-\lambda_{1}}} \cos(w_{k^*}x)\begin{pmatrix}
1\\ -1
\end{pmatrix} + O(\varepsilon).
\end{split}
\end{align}

\subsection{Numerical verification of the amplitude equation}\label{sec_num_amplitude}
In this subsection we show the validity of the approximate solution $(u,v)$ in \eqref{Eq: AMP def} by comparing it with the numerical solution of the system \eqref{simplified2}--\eqref{poisson}. We use an implicit Lax-Friedrichs finite-volume scheme to solve the system \eqref{simplified2}--\eqref{poisson}, where we use a fine mesh $\ \Delta x = \Delta t = 10^{-3} $ to ensure the numerical simulations have high accuracy. More details of the numerical scheme can be found in Appendix~\ref{sec_numerical}. For the numerical solution $u(x,t)$, at each time $t$ we define its amplitude $A_{amp}(t)$ as
\begin{align}\label{Eq: AMP def}
\begin{split}
A_{amp}(t) := \left \| u(\cdot,t) - \overline{u} \right \|_{L^\infty(\Omega)},
\end{split}
\end{align}
and we will compare it to the analytical amplitude $A(t)$ which solves \eqref{Eq: AMP}.
\begin{figure}[ht]
	\centering
	\includegraphics[scale = 0.35]{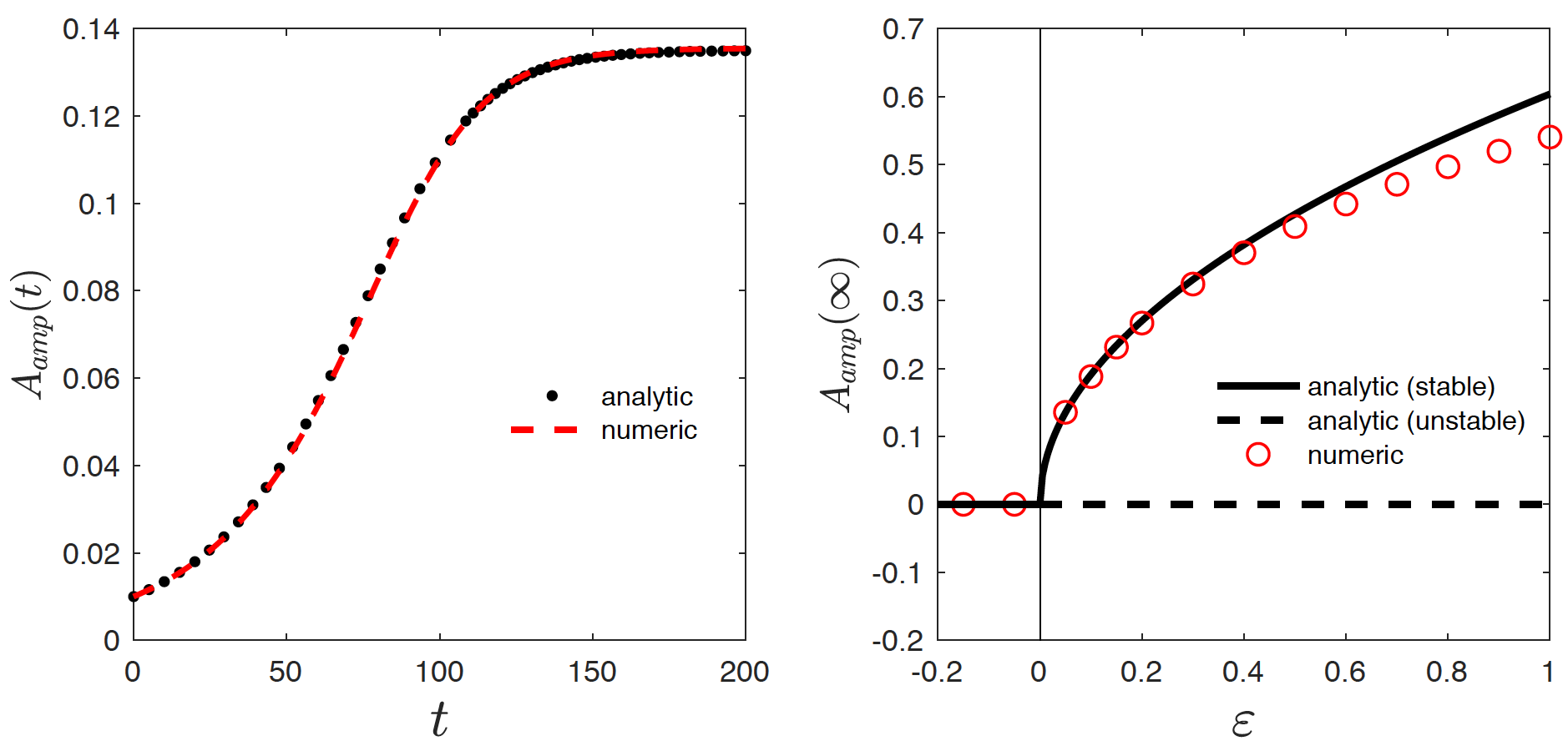}
	\caption{Comparison of the analytic amplitude equation \eqref{Eq: AMP} and the numerical solutions to the full PDE system \eqref{simplified2}--\eqref{poisson}. (Left) Comparison of the solution $A(t)$ to \eqref{Eq: AMP} and numerical amplitude $A_{amp}(t)$ when $\varepsilon = 0.05$. (Right) Comparison of the limiting analytic amplitude $A(\infty)$ (black curve) and the numerical stationary amplitude $A_{amp}(\infty)$ (red circles) for various $\varepsilon<1$, where each red circle corresponds to a numerical simulation for such $\varepsilon$.}
	\label{Fig: fig_1D_verification_amplitude}
\end{figure}\\

In the experiment, we fix the one-dimensional domain as $\Omega = [0, L]$ with $L = 2$, with the competition coefficient $a=0.2$ being in the weak competition regime. By Proposition~\ref{remark: 1D linear critical}, the critical chemotactic coefficient is $\chi^* \approx 5.28$, with the critical wavenumber $k^* = 1$. 
We choose the initial data to be small perturbations of the constant steady state \eqref{Eq: constant steady state}, given by
\begin{align}\label{Eq: ic_perturb}
\begin{split}
\begin{pmatrix}
u(x,0)\\ 
v(x,0)
\end{pmatrix} = \begin{pmatrix}
\overline{u}\\ 
\overline{v}
\end{pmatrix} + A_{0}\text{cos}\left(\frac{k^*\pi}{L}x\right)\begin{pmatrix}
1\\ 
-1
\end{pmatrix}
\end{split}
\end{align}
with the initial amplitude being $A_{0} = 10^{-2}$. We then set $\chi= \chi^* + \varepsilon$ with various choices of $|\varepsilon|<1$, so that the system is in the weakly nonlinear regime. Note that for $a=0.2$ and $L=2$, we have $\lambda_1 < 0$ and $\lambda_2 >0$, implying supercritical bifurcation behavior of \eqref{Eq: AMP} in terms of the parameter $\varepsilon$. 

For each value of $\varepsilon$, we compare the analytical solution $A(t)$ to \eqref{Eq: AMP} (solved by the ODE solver \verb ode45  built in MATLAB) to the numerical amplitude $A_{amp}(t)$ given by \eqref{Eq: AMP def}.  The left of Figure  \ref{Fig: fig_1D_verification_amplitude} shows a comparison between $A_{amp}(t)$ and $A(t)$ for $\varepsilon=0.05$ and $t\in[0,200]$, which shows that the numerical and analytical amplitude are nearly identical to each other when $\varepsilon \ll 1$, confirming the validity of our derivation in Section~\ref{sec_weakly_nonlinear}.

As $t\to \infty$, if $A(0)>0$, since $\lambda_2>0$ and $\lambda_1<0$, \eqref{Eq: AMP} gives that $A(\infty) =0$ if $\varepsilon<0$, and $A(\infty)=(-\frac{\varepsilon \lambda_2}{\lambda_1})^{-1/2}$ if $\varepsilon>0$. For a variety of $\varepsilon$, we perform numerical simulations of \eqref{simplified2}--\eqref{poisson} and record the numerical amplitude of the steady state $A_{amp}(\infty)$. (In the simulation we actually use $A_{amp}(T)$ with $T=200$, by which the numerical solution has already converged to a steady state). The comparison is shown on the right of Figure  \ref{Fig: fig_1D_verification_amplitude}, where each red circle corresponds to a numerical experiment performed at such $\varepsilon$. It shows that $A(\infty)$ and $A_{amp}(\infty)$ have excellent agreement for $\varepsilon\ll 1$, although the error becomes visible as $\varepsilon$ gets closer to 1.

\section{Numerical results in the fully nonlinear regime}\label{sec_simulations}
In this section, we numerically investigate the behavior of solutions of the system \eqref{simplified}--\eqref{poisson}. Most of the study will be focused on the one-dimensional case $\Omega=[0,L]$ with large $\chi_1,\chi_2>0$, where the constant steady state \eqref{steady_general} is linearly unstable, and we are in the fully nonlinear regime. At the end of this section, we also show two numerical simulations of the system in two dimensions in the weakly nonlinear regime and fully nonlinear regime respectively.

\subsection{Steady states in the fully nonlinear regime.} 
We first consider the system \eqref{simplified2}--\eqref{poisson} where the coefficients of the $u,v$ equation are symmetric to each other. %Recall that in Section \ref{sec_weakly_nonlinear}, we showed that if the chemotactic coefficient $\chi$ is in the weakly nonlinear regime $\chi = \chi^*+\varepsilon$ with $0<\varepsilon\ll 1$, we expect the solution to converge to the steady state \eqref{Eq: stat_sol} as $t\to\infty$, where $u,v$ are both periodic solutions with period $\frac{2L}{k^*}$, where $k^*$ is the critical wavenumber. 
We will perform numerical experiments when $\chi>\chi^*$ is in the fully nonlinear regime, i.e. $\chi-\chi^*$ is a large constant. As in Section~\ref{sec_num_amplitude}, we fix $a=0.2$ to be in the weak competition regime. In order to explore the periodic pattern of the solutions, we consider a larger domain $\Omega=[0, L]$ with $L = 30$. For such $L$ and $a$, Proposition~\ref{remark: 1D linear critical} yields that the critical chemotactic coefficient is $\chi^* \approx 3.97$, with the critical wavenumber $k^* = 9$. 
In our experiment we set $\chi = 20$, and set the initial data to be small perturbations of the constant steady state. The evolution of the densities $u$ and $v$ is shown in Figure~\ref{Fig: fig_1D_fullnonlinear_SpatialTime_profile}. The numerical solution suggests that as $t\to\infty$, $(u,v)$ converges to a steady state with a periodic pattern. The steady state is shown in Figure~\ref{Fig: fig_1D_fullnonlinear_SteadyState}(a), where the two densities exhibit a clustering effect due to the large repulsive chemotactic coefficient $\chi$. Although it appears that $u$ and $v$ almost have disjoint support, note that they are both strictly positive in  $\Omega$ by Harnack's inequality for parabolic equations.

 %our intuition is that the species with higher chemotactic sensitivity will result in the sharp local aggregation spike, whereas another species will occupy wider range of habitat. The simulations in Figure \ref{Fig: fig_1D_fullnonlinear_SteadyState} are consistent with our expectations.

\begin{figure}[ht]
	\centering
	\includegraphics[scale = 0.6]{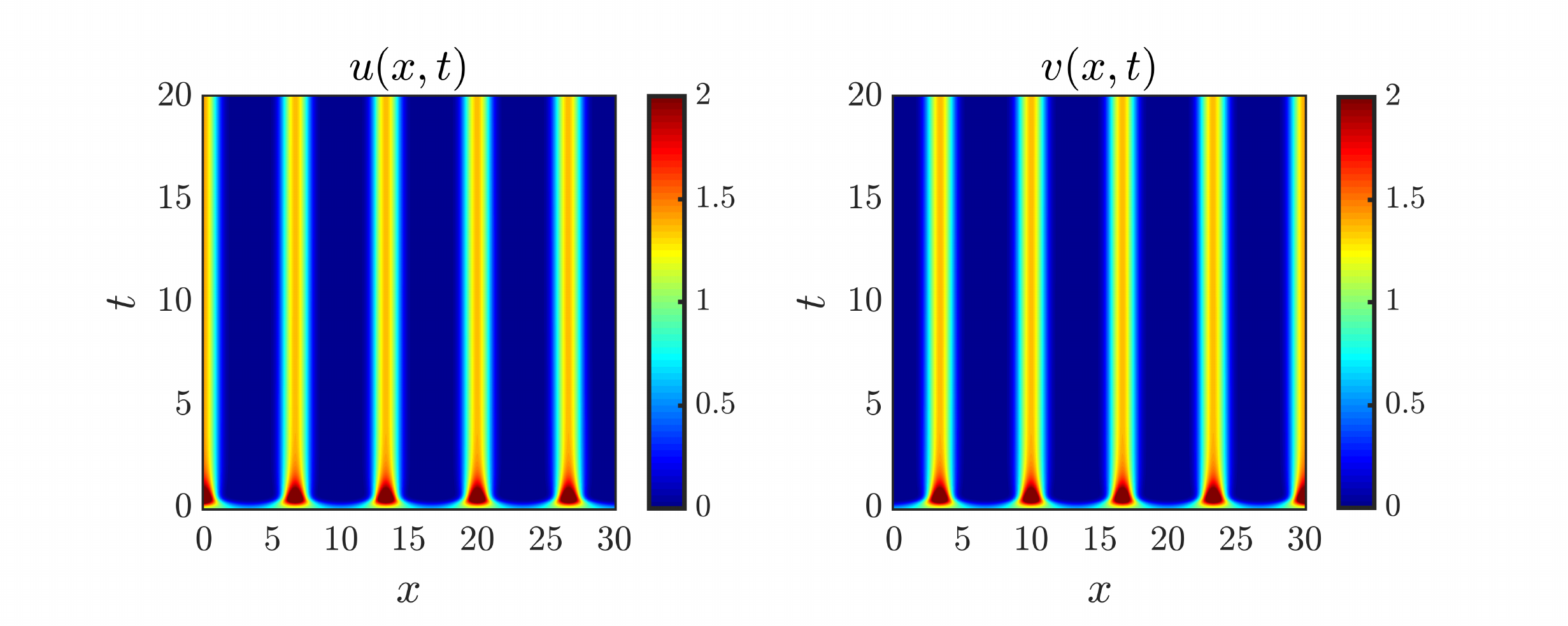}
	\caption{Time evolution of solutions in the fully nonlinear regime. For the system \eqref{simplified2}--\eqref{poisson}, we fix $a=0.2$ and $\chi=20$. Note that such $\chi$ is in the fully nonlinear regime since the critical chemotactic coefficient is $\chi^* \approx 3.97$. The initial density is a small perturbation to $(\overline u, \overline v)$.}
	%The critical chemotaxis parameter and critical wavenumber are computed as $\chi^* \approx 3.97$ and $k^* = 9$ respectively. In simulation, we use $\chi_{1} = \chi_{2} = -20$, which are far from the critical value. The initial densities are given in \eqref{Eq: ic_perturb}.}
	\label{Fig: fig_1D_fullnonlinear_SpatialTime_profile}
\end{figure}

Next we consider the original system \eqref{simplified}--\eqref{poisson} with asymmetric chemotactic coefficients $0<\chi_1<\chi_2$. All the other coefficients for the two equations are symmetric, and remain the same as above (i.e. $d_1=d_2=1$, $a_1=a_2=0.2$, $b_1=b_2=1$). When $\chi_2>\chi_1>0$, $v$ is more sensitive than $u$, thus heuristically one expects that the cluster for $v$ should have a thinner width and higher spike. When the initial data is a small perturbation to the constant solution, the steady state is shown in Figure~\ref{Fig: fig_1D_fullnonlinear_SteadyState}(b), which is consistent with our expectations.

\begin{figure}[ht]
	\centering
	\includegraphics[scale = 0.5]{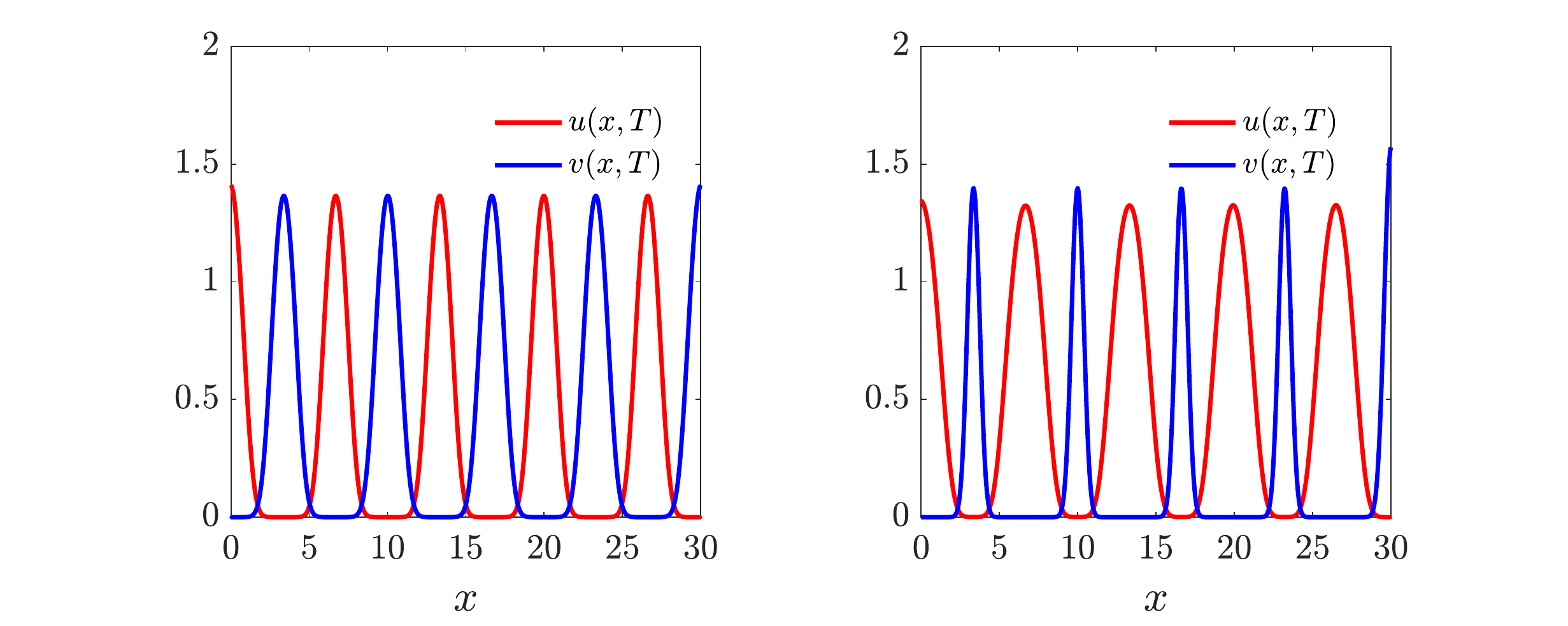}\\
	(a)\hspace{4.5cm}(b)
	\caption{(a)  Steady state for two species with the same chemotactic sensitivity coefficients $\chi=20$. (b) Steady state when $u$ and $v$ have different chemotactic sensitivity coefficients $\chi_1=20$ and $\chi_2=40$. %and different chemotactic sensitivity coefficients (right). (Left) Symmetric steady states with $\chi=20$. (Right) Asymmetric steady states are formed when $u$ and $v$ have different chemotactic sensitivity coefficients, where $\chi_{1} = 20$ and $\chi_{2} = 40$. %Initial densities are given in Eq. (\ref{Eq: ic_perturb}).
	}
	\label{Fig: fig_1D_fullnonlinear_SteadyState}
\end{figure}

\subsection{Traveling wave for segregated initial data.} In this part, we consider initial densities $(u_0,v_0)$ that are segregated, i.e., $\text{supp}\,u_0  = [s_2,L]$ and $\text{supp}\,v_0 = [0,s_1]$, where $0<s_1<s_2<L$. To see the front propagation clearly, we consider a larger domain $\Omega=[0,L]$ with $L=100$.  The initial densities of $u$ and $v$ are chosen to be uniform in their supports and have mass one, i.e.
\begin{align}\label{Eq: 1D, initial density, locally support, disjoint}
\begin{split}
u_0(x) = \frac{\mathbf{1}_{[s_2,L]}(x)}{L-s_2}, \quad
v_0(x) = \frac{\mathbf{1}_{[0,s_1]}(x)}{s_1},
\end{split}
\end{align}
where we fix $s_1=10, s_2=90$. In this experiment, we focus on the strong competition regime $a_1=a_2=2$, with $d_i$ and $b_i$ the same as in last part.

%grow without local competing (resource competition and local gradients of repellents) until they interact at the central regime, naturally it makes sense that two species evolve into two static steady states when they suffer the same force of repulsion (see Figure \ref{Fig: fig_1D_traveling_wave_profiles}). 

%In the case of asymmetric chemotactic effect, there exists constant-velocity traveling wave (see Figure \ref{Fig: fig_1D_traveling_wave_SpatialTime}). Species $u$ (high chemotactic sensitivity) is chased by the species $v$ (low chemotactic sensitivity), the front wave profiles are shown in Figure \ref{Fig: fig_1D_traveling_wave_profiles}.

\begin{figure}[ht]
	\centering
	\includegraphics[scale = 0.55]{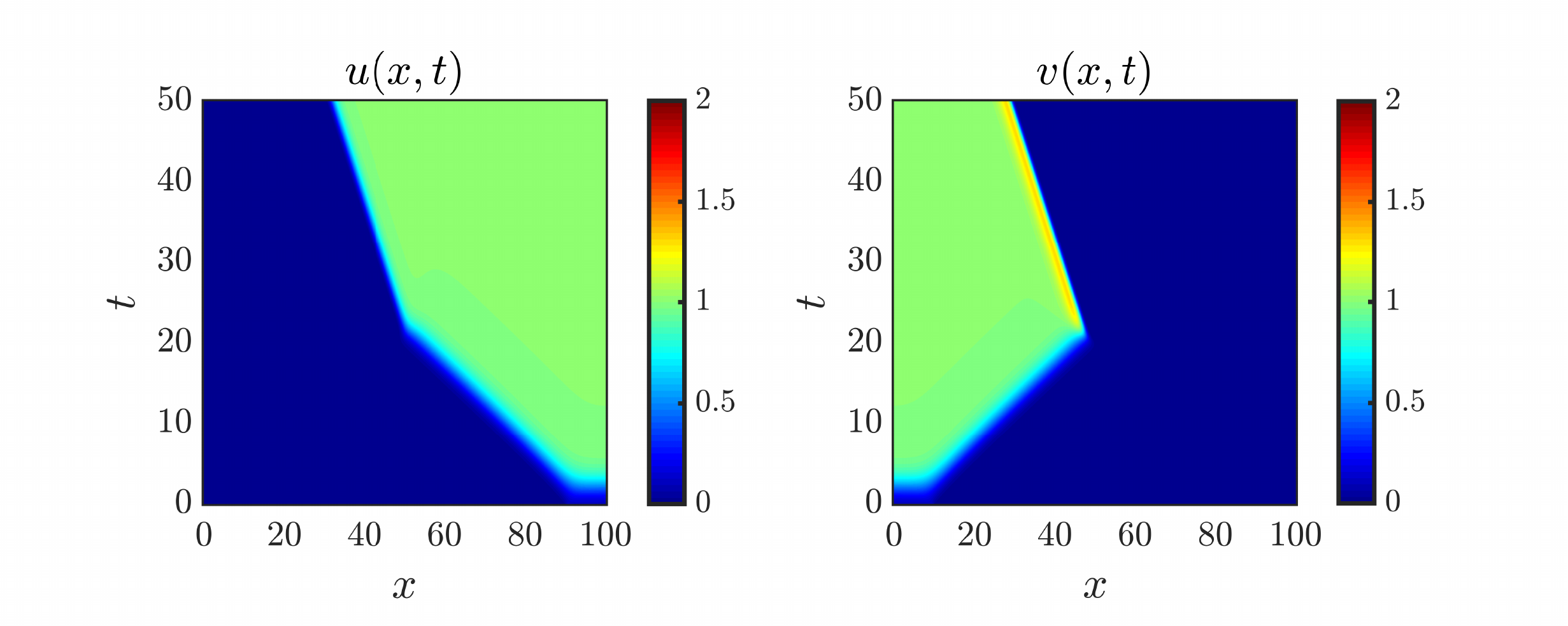}
	\caption{Spatio-temporal evolution of the solution $(u,v)$ with segregated initial data \eqref{Eq: 1D, initial density, locally support, disjoint}, and asymmetric chemotactic coefficients $\chi_1=20$, $\chi_2=80$, and $a=2$ is in the strong competition regime.  After $t \approx 20$, the solution $(u,v)$ forms a traveling wave, which propagates towards the left at a constant velocity.}
	\label{Fig: fig_1D_traveling_wave_SpatialTime}
\end{figure}

\begin{figure}[ht]
	\centering
	\includegraphics[scale = 0.55]{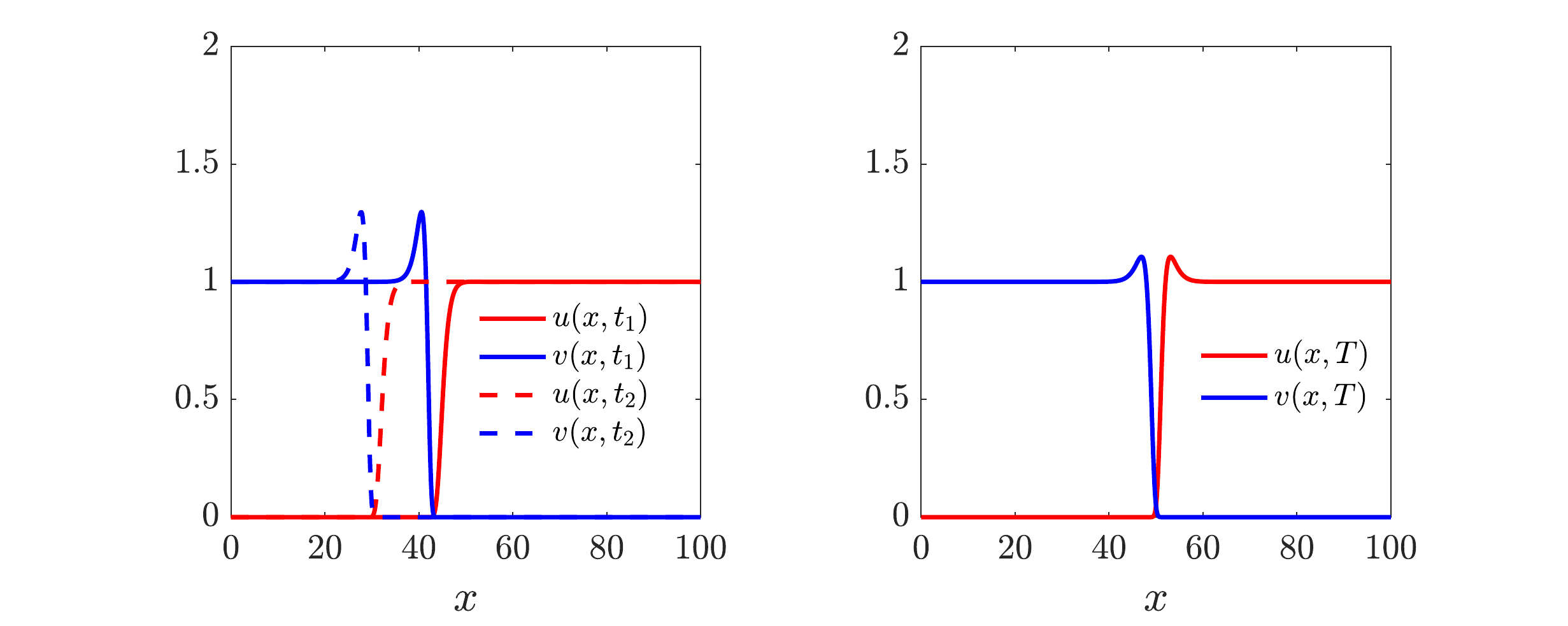}
	\caption{The solution $(u,v)$ with segregated initial data \eqref{Eq: 1D, initial density, locally support, disjoint} and strong competition $a=2$  at different times ($t_{1} = 30$ and $t_{2} = 50$) after early transient profiles.  (Left) For $\chi_1=20$, $\chi_2=80$, the solution $(u,v)$ converges to a traveling wave with negative velocity.  (Right) For $\chi_1=\chi_2=20$, the solution $(u,v)$ converges to a segregated steady state, which can be seen as a traveling wave with velocity 0.}
	\label{Fig: fig_1D_traveling_wave_profiles}
\end{figure}

Since the initial support of $u_0$ and $v_0$ are far apart, for some time the dynamics for both $u$ and $v$ are almost identical to the reaction-diffusion equation for a single species $u_t = u_{xx}+u(1-u)$, where both densities propagate towards the center as traveling waves. As their fronts approach each other, the chemotactic and competition effects start to take place. Numerically, we observe that in the long run, the pair of solutions $(u,v)$ converges to a traveling-wave solution $(\tilde u(x-ct), \tilde v(x-ct))$ with velocity $c$, where $c>0$ if $\chi_1>\chi_2$, and $c<0$ if $\chi_1<\chi_2$. Heuristically, this can be understood as the more sensitive species experiences stronger repulsion and keeps receding (thus becomes invaded by the other species), forming a traveling wave. We remark that when the system \eqref{simplified}--\eqref{poisson} only has one chemical that is emitted by both $u$ and $v$, existence of traveling wave was recently established by \cite{issa2020traveling}.

For asymmetric chemotactic coefficients $\chi_1=20, \chi_2=80$,  Figure~\ref{Fig: fig_1D_traveling_wave_SpatialTime} gives the spatio-temporal evolution for both $u$ and $v$. It shows that once the two densities approach each other (at $t\approx20$), the pair of solutions $(u,v)$ forms a traveling wave with negative velocity as expected.

Figure~\ref{Fig: fig_1D_traveling_wave_profiles} shows the traveling wave profiles for two different choices of chemotactic coefficients. The left figure corresponds to the choice $\chi_1=20$ and $\chi_2=80$, where the traveling wave propagates towards the left. In the right figure we consider identical chemotactic strengths $\chi_1=\chi_2=20$, where the solution converges to a segregated steady state as $t\to\infty$.

\subsection{Front propagation for compactly-supported initial data.} 
In this part, we consider initial densities $(u_0,v_0)$ that are compactly supported, i.e., $\text{supp}\,u_0  = I_1$ and $\text{supp}\,v_0 = I_2$, where $I_2 \subset I_1 \subset [0,L]$. We fix $L=100$, and set $I_1 = [45,55]$ and $I_2=[48,52]$. The initial densities of $u$ and $v$ are chosen to be uniform in their supports, i.e.
\begin{align}\label{Eq: 1D, initial density, locally support}
\begin{split}
u_0(x) = \frac{\mathbf{1}_{I_1 }(x)}{|I_1|}, \quad
v_0(x) = \frac{\mathbf{1}_{I_2}(x)}{|I_2|},
\end{split}
\end{align}
so that the mass of $u_0$ and $v_0$ are both one. 

For such initial data, our first experiment deals with the system \eqref{simplified2}--\eqref{poisson} where  $u$ and $v$ have symmetric cofficients, and we fix $\chi=20$, $a=0.2$. The numerical simulation is shown in Figure~\ref{Fig: fig_1D_rippling_wave}, where the first two figures show the spatio-temporal evolution for both $u$ and $v$, and the third figure gives a snapshot of the two densities at time $t=10$.

\begin{figure}[ht]
	\centering

	\hspace*{-1.2cm}\includegraphics[scale = 0.5]{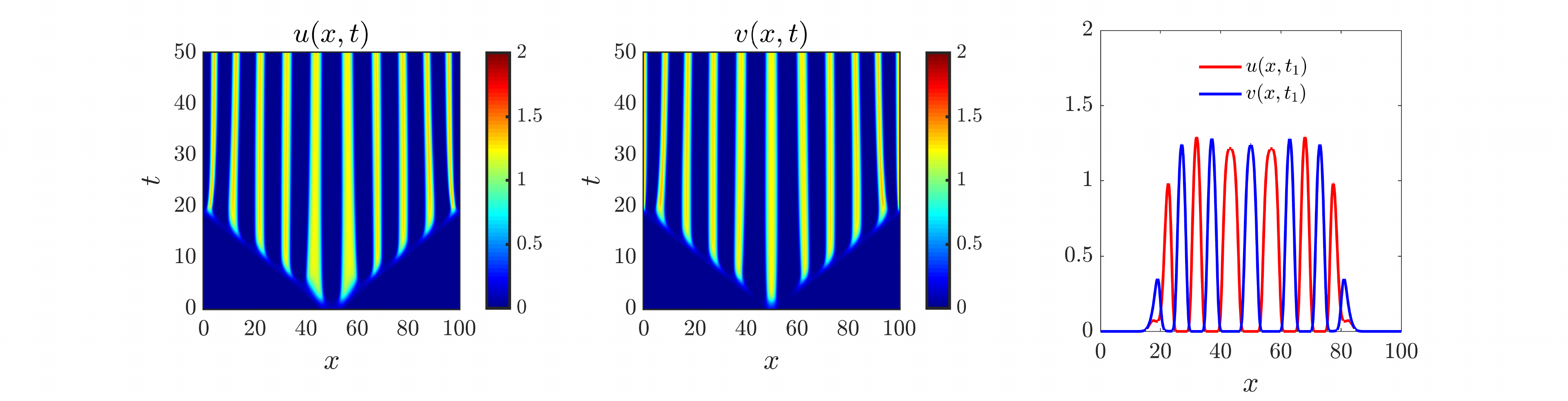}\\
	\caption{Front propagation with compactly-supported initial densities given by \eqref{Eq: 1D, initial density, locally support}.  Here the parameters are $\chi_1=\chi_2=20$, with weak competition $a=0.2$. The first two figures show the spatio-temporal evolution for both $u$ and $v$, and the third figure gives a snapshot of the two densities at time $t=10$. 
	\label{Fig: fig_1D_rippling_wave}}
\end{figure}

Interestingly, as $t$ increases, both densities propagate outwards, with new clusters of $u$ and $v$ forming alternatively over time: once the outmost cluster of $u$ grows to a certain height, a new small cluster of $v$ always starts to form outside of it, and vice versa.
Heuristically, this phenomenon is due to a combination of the diffusion and strong chemotactic repulsion effects. Note that when a cluster of $u$ dominates the front,  there are still some tiny densities of both $u$ and $v$ outside the front, since $u,v$ are both strictly positive at any $t>0$ due to diffusion. As long as the front is dominated by $u$, the tiny population of $v$ outside of the front experiences a large outward drift velocity caused by the strong chemotactic repulsion, which makes it propagate outwards faster than $u$. This additional drift helps $v$ dominate the space outside of the front, where a cluster of $v$ starts to form as the small density $v$ grows exponentially.

In the weak competition regime, when the two species has asymmetric chemotactic coefficients $\chi_1\neq \chi_2$, we observe a similar propagation pattern as above (the numerical simulations are omitted),  except that the clusters of one species are thinner and taller than the other, like in Figure~\ref{Fig: fig_1D_fullnonlinear_SteadyState}(b). Regardless of the pattern, recall that in Theorem~\ref{no_extinction} we proved that for any $\chi_1,\chi_2>0$, neither species would become extinct in the weak competition regime.

\begin{figure}[ht]
	\centering
	\hspace*{-1.2cm}\includegraphics[scale = 0.5]{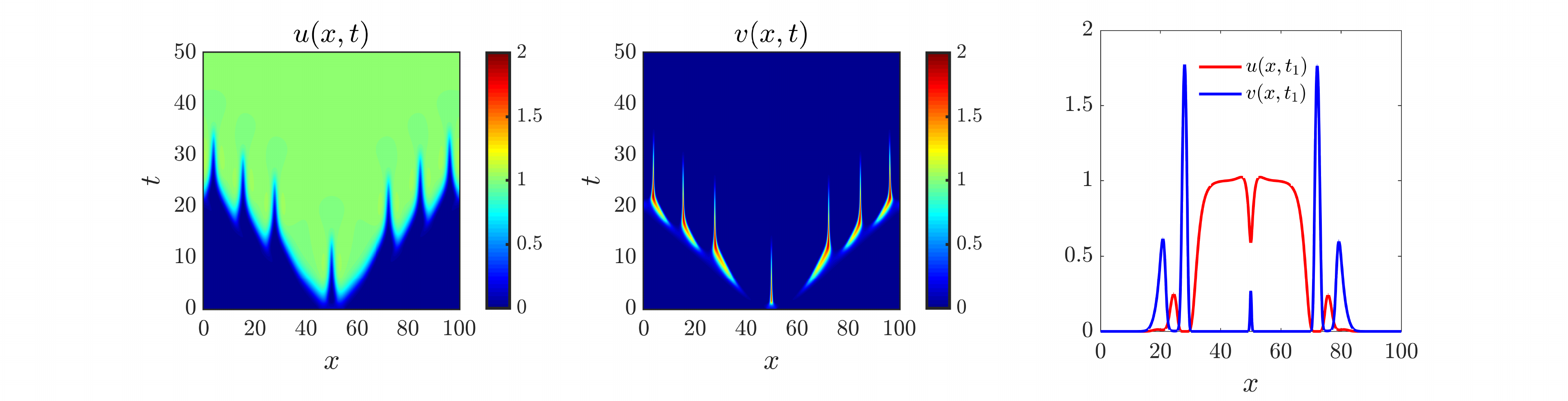}\\
	\caption{Front propagation with compactly-supported initial densities given by \eqref{Eq: 1D, initial density, locally support}.  Here the parameters are $\chi_1=10, \chi_2=100$, with strong competition $a=2$. The first two figures show the spatio-temporal evolution for both $u$ and $v$, and the third figure gives a snapshot of the two densities at time $t=10$. 
	\label{Fig: fig_1D_rippling_wave2}}
\end{figure}

Next we move on to the strong competition regime,  where we set $a_1=a_2=2$ with asymmetric chemotactic coefficients $\chi_1=10, \chi_2=100$. The numerical simulation is shown in Figure~\ref{Fig: fig_1D_rippling_wave2}, where we find that as $t\to\infty$, the more sensitive species $v$ is wiped out by the species $u$.

\subsection{Numerical simulations in two dimensions}\label{num_2d}
In this subsection, we perform two numerical simulations in 2D on the system \eqref{simplified2}--\eqref{poisson} where the coefficients of the $u,v$ equation are symmetric to each other. The first example is done in the weakly nonlinear regime with weak competition, and the initial data is chosen to be close to the constant steady state. 
The second example is done in the fully nonlinear regime with strong competition, and the initial data $u_0, v_0$ are localized Gaussian distributions. For both examples, our domain is a square $\Omega =[0, L]^{2}$ with $L = 30$, which is uniformly divided into a mesh of size $\Delta x = \Delta y = 0.1$, and the time step size is $\Delta t = 0.05$. Our numerical scheme is described in Appendix~\ref{sec_numerical}.

%The computational cost in solving higher dimension system is always high. 
%We are interested in the long time (steady state) patterns, therefore we need a convergence terminator such that the simulation will stop if it reaches the termination criteria. For example, we consider following "distance",
%\begin{align}\label{Eq: error norm 2D}
%\begin{split}
%\tau^{n} = \left \| U^{n+1} - U^{n} \right \|_{F}
%\end{split}
%\end{align}
%where $U^{n}$ is the numerical solution of $u(x,y,t)$ at time $t = t^{n}$, distance is measured by Frobenius norm $\left \|. \right \|_{F}$. Since the numerical scheme (see Eq. (\ref{Eq: scheme final})) is locally second order accuracy in space and time, it makes sense that we design the termination criteria as $\tau ^{n}\ \leq\ 10^{-6}$, which is a good enough based on the mesh size we use.\\
%\\

\textbf{Weakly nonlinear regime in 2D.} 
When $\Omega=[0,L]^2$, the Neumann eigenvalue problem \eqref{eigenvalue} has a family of eigenfunctions $\{\varphi_{k,j}\}_{k,j\in\mathbb{N}}$ where $\varphi_{k,j}=\cos(\frac{k\pi}{L}x)\cos(\frac{j\pi}{L}y)$, with corresponding eigenvalue $\lambda_{k,j}=(\frac{k\pi}{L})^2 + (\frac{j\pi}{L})^2$. Using this family of eigenvalues $\{\lambda_{k,j}\}_{k,j\in\mathbb{N}}$, we can apply Proposition~\ref{remark: 1D linear critical} to obtain the linear stability criteria of the constant steady state \eqref{Eq: constant steady state}. Namely, when $L=30, a=0.5$, Proposition~\ref{remark: 1D linear critical} shows that the critical chemotactic coefficient is $\chi^*  \approx 3.7$, with
 the critical wavenumber $(k^*, j^*) \in \mathbb{N}_{+}^{2}$ given by $k^* = 7$ and $j^* = 2$.

\begin{figure}[ht]
	\centering
	\includegraphics[scale = 0.45]{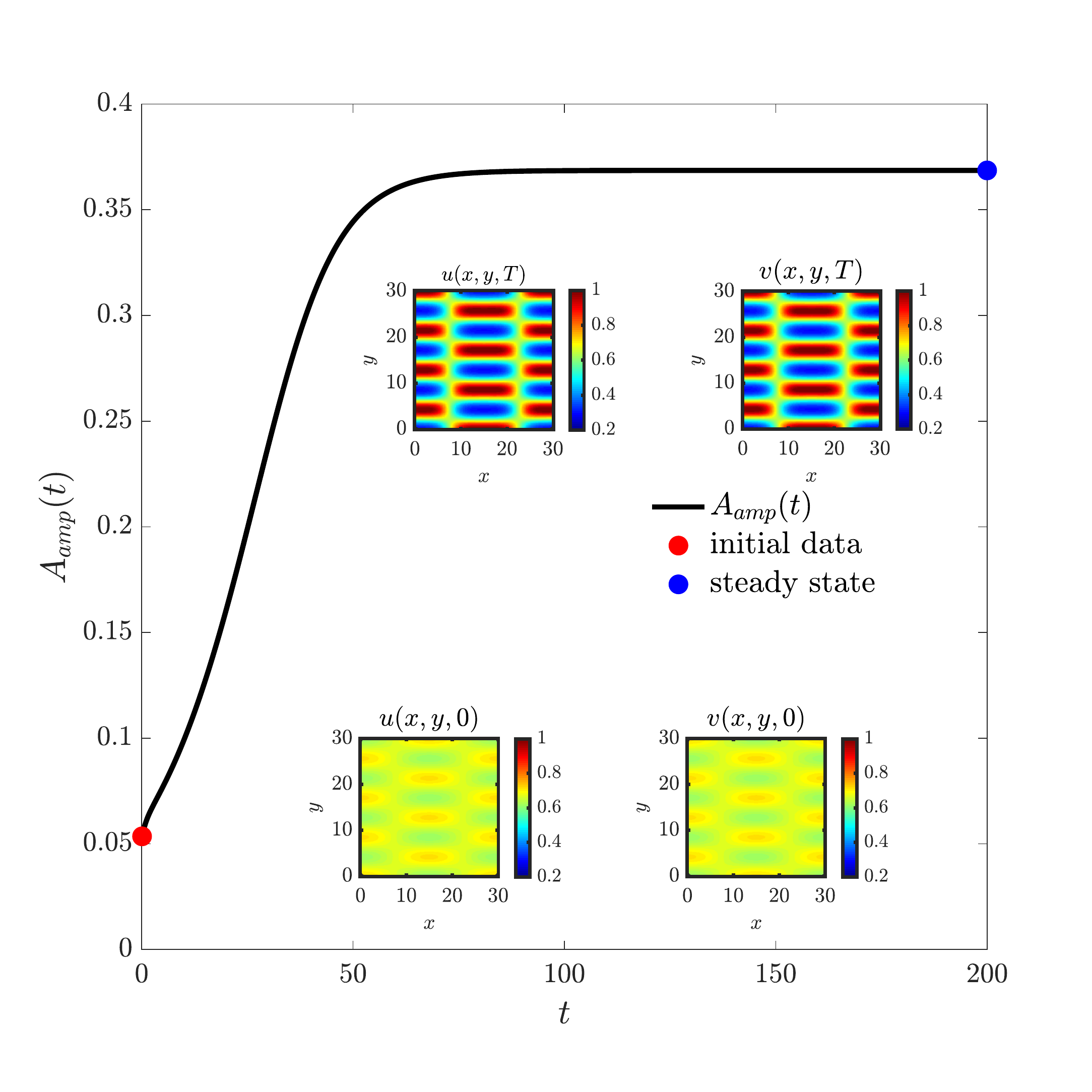}
	\caption{Time evolution of the amplitude $A_{amp}(t)$ in the weakly nonlinear regime for a 2D simulation. }
	\label{Fig: fig_2d_amplitude}
\end{figure}

In our numerical simulation we fix $\chi = 4.7$, which is slightly larger than the critical value $\chi^*$. 
Let us take the initial initial densities to be a small perturbation of the constant steady state \eqref{Eq: constant steady state}, given by
\begin{align}\label{Eq: ic_perturb_2d}
\begin{split}
\begin{pmatrix}
u(x,y, 0)\\ 
v(x,y, 0)
\end{pmatrix} = \begin{pmatrix}
\overline{u}\\ 
\overline{v}
\end{pmatrix} + A_{0}\cos\left(\frac{k^*\pi}{L}x\right)\cos\left(\frac{j^{*}\pi}{L}y\right)\begin{pmatrix}
1\\ 
-1
\end{pmatrix},
\end{split}
\end{align}
where $A_{0} = 0.05$ is the amplitude of initial perturbation.   In Figure~\ref{Fig: fig_2d_amplitude} we track the evolution of the numerical amplitude $A_{amp}(t)$ defined in \eqref{Eq: AMP def}, and also show the spatial patterns of $u$ and $v$ at the initial time and after a long time.
 
\textbf{Wave propagation in 2D.} In this example, we consider the system \eqref{simplified2}--\eqref{poisson} with $a=2$ (strong competition) and $\chi=100$, where the system is in the fully nonlinear regime. Initial densities are given in the form of Gaussian distributions,
\begin{align}\label{Eq: ic_gaussian_2d}
\begin{split}
u(x,y,0) \sim \mathcal{N}\left (\begin{pmatrix}
\frac{L}{2}\\ 
\frac{L}{2}
\end{pmatrix},\  \begin{pmatrix}
\sigma_{1}^{2} & 0\\ 
 0 & \sigma_{1}^{2} 
\end{pmatrix} \right),\quad 
v(x,y,0) \sim \mathcal{N}\left (\begin{pmatrix}
\frac{L}{2}\\ 
\frac{L}{2}
\end{pmatrix},\  \begin{pmatrix}
\sigma_{2}^{2} & 0\\ 
 0 & \sigma_{2}^{2} 
\end{pmatrix} \right)
\end{split}
\end{align}
where $\sigma_{1}^{2} = \frac{1}{4}$ and $\sigma_{2}^{2} = \frac{1}{9}$. The numerical simulations of $u$ and $v$ are illustrated in Figure \ref{Fig: fig_2d_gaussian_propagation}. We observe that the two densities propagate outwards in ring patterns for a short time, and more complicated patterns start to form once the densities reach the boundary of the domain. Finally, species $u$ occupies most of the space in the square domain, with species $v$ suppressed at the center of domain with a high concentration.

 \begin{figure}[ht]
	\centering
	\includegraphics[scale = 0.45]{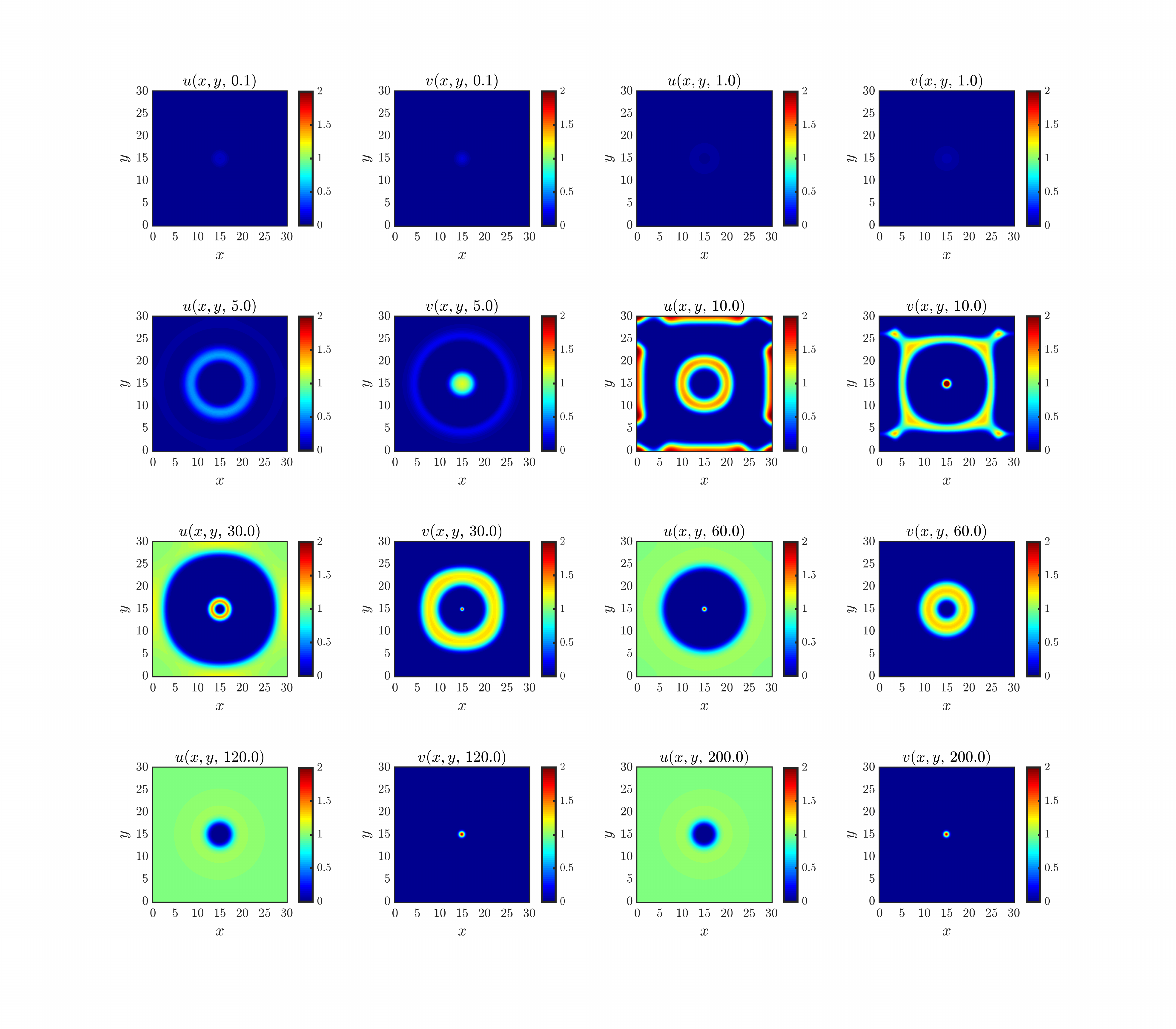}
	\caption{
Time evolution of two species with Gaussian distributed initial data in the fully nonlinear regime $\chi=100$, with strong competition $a=2$. The images show the numerical solutions at $t = 0.1, 1, 5, 10, 30, 60, 120$ and $200$. When $t > 120$, the system is approaching the stationary state, the density profiles of two species at $t = 120$ and $t = 200$ are nearly identical.\label{Fig: fig_2d_gaussian_propagation}
}
	\end{figure}

\begin{appendix}
\section{Numerical scheme}\label{sec_numerical}
In this section we describe the numerical scheme we use for the numerical simulations in Section~\ref{sec_num_amplitude} and Section~\ref{sec_simulations}. 
We will only describe the numerical scheme in two dimensions, since the numerical scheme in one dimension is similar and easier. Our numerical method is a semi-implicit finite volume scheme. % based on ([Ref]) \textcolor{red}{(I'll take care of the reference)}.
% Implicit schemes in general tend to produce reliable results for large time step sizes, however, it is difficult to generate higher order scheme in pratical. 
The system \eqref{simplified2}--\eqref{poisson} is solved on the square domain $\Omega = [0,L]^{2} \subset \mathbb{R}^{2}$ and we uniformly divide it into $N^{2}$ cells $\{\mathcal{C}_{j,k}\}$ using a Cartesian mesh, where
\begin{align}\label{Eq: cell}
\begin{split}
\mathcal{C}_{j,k} := \big[x_{j-\frac{1}{2}},x_{j+\frac{1}{2}}\big] \times \big[y_{k-\frac{1}{2}},y_{k+\frac{1}{2}}\big] \quad\text{ for } j,k = 1,2,...,N.
\end{split}
\end{align}
Each cell has size $|\mathcal{C}_{j,k}| = \Delta x \Delta y$ and the cell $\mathcal{C}_{j,k}$ is centered at $(x_{j},y_{k})$, where $\Delta x = x_{j+\frac{1}{2}} - x_{j-\frac{1}{2}}$ and $\Delta y = y_{k+\frac{1}{2}} - y_{k-\frac{1}{2}}$ for all $j, k = 1,2,...N$ respectively. We also discretize the time interval $[0,T]$ as $t^{n} = n\Delta t$, where $n = 0,1,...,\frac{T}{\Delta t}$.

The computed quantities are the cell averages. For a density $\rho(x,y,t)$ that is either $u$ or $v$, let us denote by $\overline{\rho}_{j,k}(t)$ the spatial average of $\rho(t)$ over the cell $\mathcal{C}_{j,k}$, i.e.,
\begin{align}\label{Eq: cell avg}
\begin{split}
\overline{\rho}_{j,k}(t) = \frac{1}{|\mathcal{C}_{j,k}|}\iint_{\mathcal{C}_{j,k}} \rho(x,y,t)\ dx dy.
\end{split}
\end{align}
 Note that the chemotaxis term gives the $\rho$ equation a flux $\eta = \chi \rho \nabla c$, where $c$ is $c[v]$ if $\rho=u$, and $c=c[u]$ if $\rho=v$.   Integrating \eqref{simplified2} over a spatial-temporal cell $ \mathcal{C}_{j,k}\times [t^{n},t^{n+1}]$, we arrive at the fully discrete, fully implicit scheme
\begin{align}\label{Eq: scheme}
\begin{split}
\frac{\overline{\rho}_{j,k}^{n+1} - \overline{\rho}_{j,k}^{n}}{\Delta t} &=  - \left( \frac{\eta_{j+\frac{1}{2},k}^{(x)} - \eta_{j-\frac{1}{2},k}^{(x)}}{\Delta x}   + \frac{\eta_{j,k+\frac{1}{2}}^{(y)} - \eta_{j,k-\frac{1}{2}}^{(y)}}{\Delta y} \right)\\ &+
\left(\frac{\overline{\rho}_{j-1,k}^{n+1} - 2\overline{\rho}_{j,k}^{n+1} + \overline{\rho}_{j+1,k}^{n+1}}{\Delta x^2} + \frac{\overline{\rho}_{j,k-1}^{n+1} - 2\overline{\rho}_{j,k}^{n+1} + \overline{\rho}_{j,k+1}^{n+1}}{\Delta y^2}\right) + f^{n+1}_{j,k}.
\end{split}
\end{align}
Here $\eta_{j\pm \frac{1}{2},k}^{(x)}$ and $\eta_{j,k\pm \frac{1}{2}}^{(y)}$ are the numerical fluxes on cell boundaries, and $f^{n+1}_{j,k}$ are the nonlinear terms due to competition. The numerical fluxes on cell boundaries are computed using the Lax-Friedrichs numerical flux: for example, the numerical flux $\eta_{j+\frac{1}{2},k}^{(x)}$ is estimated as
\begin{align}\label{Eq: flux}
\begin{split}
\eta_{j+\frac{1}{2},k}^{(x)}(\overline{\rho}^{n+1}_{j,k},\overline{\rho}^{n+1}_{j+1,k}) = \frac{1}{2}\left(\phi(\overline{\rho}^{n+1}_{j,k}) + \phi(\overline{\rho}^{n+1}_{j+1,k})\right) -\frac{\Delta x}{2 \Delta t}(\overline{\rho}^{n+1}_{j+1,k} - \overline{\rho}^{n+1}_{j,k}) 
\end{split}
\end{align}
where
\begin{align}\label{Eq: flux supp}
\begin{split}
\phi(\overline{\rho}^{n+1}_{j,k}) = \chi\ \overline{\rho}^{n+1}_{j,k} \left(\frac{\overline c_{j+1,k}^{n} - \overline c_{j-1,k}^{n}}{2\Delta x}\right).
\end{split}
\end{align}
The average chemical density $\{\overline c_{j,k}^n\}$ on the cells are recovered from either $\{\overline u_{j,k}^n\}$ or $\{\overline v_{j,k}^n\}$ by solving the equation \eqref{poisson} with a standard central difference scheme, which is implemented with an efficient iterative method.

Since the terms $f^{n+1}_{j,k}$ in \eqref{Eq: scheme} is a nonlinear function of $\overline u^{n+1}_{j,k}$ and $\overline v^{n+1}_{j,k}$, to avoid solving a nonlinear system of equations at each iteration, we linearize it as% to improve computational efficiency, this can be viewed as the linearization. In doing so, we write $f^{n+1}_{j,k}$ as linearized approximation,
\begin{align}\label{Eq: linearization}
\begin{split}
f_{j,k}^{n+1} \approx f(\overline{u}_{j,k}^{n},\overline{v}_{j,k}^{n}) + \frac{\partial f}{\partial u}\Big |_{(\overline{u}_{j,k}^{n},\overline{v}_{j,k}^{n})}(\overline{u}^{n+1}_{j,k} - \overline{u}^{n}_{j,k}) + \frac{\partial f}{\partial v}\Big |_{(\overline{u}_{j,k}^{n},\overline{v}_{j,k}^{n})}(\overline{v}^{n+1}_{j,k} - \overline{v}^{n}_{j,k})
\end{split}
\end{align}
 to lower the computation costs, and note that the linearization does not affect the first order scheme accuracy. 
 
From our discussion above, for $j,k=1,\dots,N$ and $\rho=u,v$, \eqref{Eq: scheme} becomes a system of linear equations with the unknowns being $\overline u_{j,k}^{n+1}$ and $\overline v_{j,k}^{n+1}$. To reduce the computational cost, in practice we implement an incomplete alternating directional implicit (ADI) method, which we describe below. To shorten the notations, let us introduce the following standard spatial difference operators
\begin{align}\label{Eq: opt spatil}
\begin{split}
\mathcal{D}_{x}(\overline{\rho}_{j,k}) := \frac{\eta_{j+\frac{1}{2},k}^{(x)} - \eta_{j-\frac{1}{2},k}^{(x)}}{\Delta x}\quad ;\quad \mathcal{D}_{xx}(\overline{\rho}_{j,k}) := \ \frac{\overline{\rho}_{j-1,k} - 2\overline{\rho}_{j,k} + \overline{\rho}_{j+1,k}}{\Delta x^2}.
\end{split}
\end{align}
We define $\mathcal{D}_{y}(.)$ and $\mathcal{D}_{yy}(.)$ in the same way.  Thus by moving all unknowns of \eqref{Eq: scheme}  (at time step $n+1$) to the left-hand side, we can rearrange the terms so that  can be suppressed to a compact form in terms of operators we defined above,
\begin{align}\label{Eq: compact scheme}
\begin{split}
\overline{\rho}^{n+1} - \left( (\mathcal{D}_{xx} - \mathcal{D}_{x}) + (\mathcal{D}_{yy} - \mathcal{D}_{y})+ \frac{d}{d\rho}f(\overline{\rho}^{n})\right) \overline{\rho}^{n+1}\Delta t = \overline{\rho}^{n} + f(\overline{\rho}^{n})\Delta t - \frac{d}{d\rho}f(\overline{\rho}^{n})\overline{\rho}^{n}\Delta t,
\end{split}
\end{align}
where the right hand side only contains terms that are known (at the time step $n$).

Let us split the operators into different directions, denoted by
\begin{align}\label{Eq: scheme operators}
\begin{split}
\mathcal{A}^{(x)} := \mathcal{D}_{xx} - \mathcal{D}_{x},\ \mathcal{A}^{(y)} := \mathcal{D}_{yy} - \mathcal{D}_{y},\ \mathcal{A}^{'} := \frac{d}{d\rho}f(\overline{\rho}^{n}).
\end{split}
\end{align}
Let us perform the imperfect factorization of of \eqref{Eq: compact scheme}:
\begin{align}\label{Eq: scheme factor}
\begin{split}
\left ( \mathcal{I} - (\mathcal{A}^{(x)} + \frac{1}{2}\mathcal{A}^{'})\Delta t \right )\left(\mathcal{I} - (\mathcal{A}^{(y)} + \frac{1}{2}\mathcal{A}^{'})\Delta t \right)\overline{\rho}^{n+1} + O(\Delta t^2) = \text{RHS of \eqref{Eq: compact scheme}},
\end{split}
\end{align}
where $\mathcal{I}$ is the identity operator. Since the numerical scheme \eqref{Eq: scheme} is a first order scheme, we do not lose any accuracy from the factorization procedure. Omitting the $O(\Delta t^2)$ terms in \eqref{Eq: scheme factor}, we can then write numerical integration process from $t^{n}$ to $t^{n+1}$ in two steps,
\begin{align}\label{Eq: scheme final}
\begin{split}
\left( \mathcal{I} - (\mathcal{A}^{(x)} + \frac{1}{2}\mathcal{A}^{'})\Delta t\right)\rho^{*} & = \text{RHS of \eqref{Eq: compact scheme}}\\
\left(\mathcal{I} - (\mathcal{A}^{(y)} + \frac{1}{2}\mathcal{A}^{'})\Delta t \right)\overline{\rho}^{n+1} & = \rho^{*} 
\end{split}
\end{align}
where $\rho^{*}$ is an intermediate solution. Finally, we use the Thomas algorithm (tridiagonal matrix algorithm)  to solve the linear systems in \eqref{Eq: scheme final}. %It's easy to adapt the 2D numerical scheme (see Eq. (\ref{Eq: scheme})) to 1D system (\ref{model1}).

\end{appendix}

\bibliographystyle{abbrv}
\bibliography{references}

\vspace{1cm}
\begin{tabular}{l}
\textbf{Guanlin Li}\\
{Interdisciplinary Graduate Program in Quantitative Biosciences} \\
{Georgia Institute of Technology} \\
{Atlanta, GA 30332, USA} \\[0.2cm]
{and} \\[0.2cm]
{School of Physics} \\
{Georgia Institute of Technology} \\
{Atlanta, GA 30332, USA} \\
{Email: guanlinl@gatech.edu} \\ \\
\textbf{Yao Yao} \\
 {School of Mathematics}\\
 {Georgia Institute of Technology} \\
{Atlanta, GA 30332, USA} \\
 {Email: yaoyao@math.gatech.edu}\\
\end{tabular}
\end{document}